\documentclass[a4paper,11pt,fullpage]{article}
\usepackage[english]{babel}

\usepackage[left=1in,right=1in,top=1in,bottom=1in]{geometry}

\usepackage{amsmath}
\usepackage{amsthm}
\usepackage{amsfonts}
\usepackage{amssymb}
\usepackage{amsxtra}
\usepackage{latexsym}
\usepackage{ifthen}
\usepackage{cite}
\usepackage{esint}
\usepackage[cp1250]{inputenc}

\usepackage{epic}
\usepackage{eepic}
\usepackage{stmaryrd}

\DeclareMathOperator*{\osc}{osc}

\numberwithin{equation}{section}
\newtheorem{theorem}{Theorem}[section]
\newtheorem{lemma}{Lemma}[section]
\newtheorem{remark}{Remark}[section]

\newtheorem{propo}{Proposition}[section]

\def\XXint#1#2#3{{\setbox0=\hbox{$#1{#2#3}{\int}$}
     \vcenter{\hbox{$#2#3$}}\kern-.5\wd0}}

\begin{document}

\title{On the  H\"{o}lder continuity of signed solutions to doubly nonlinear  parabolic equations in the mixed degenerate/singular  cases}

\author{ Igor I. Skrypnik
 }
\maketitle

  \begin{abstract}

We prove the H\"{o}lder continuity of sign-changing solutions to the equation of the type
$$\frac{\partial}{\partial t}\big(|u|^{q-1} u\big)- div\Big(|D u|^{p-2}\,D u\Big)=0,$$ 
where  numbers $p$, $q$ satisfy  the conditions
$$0<q<p-1\quad 
\text{and}\quad p<2,$$
or
$$q>p-1\quad\text{and}\quad p>2.$$ 
Our proof uses new versions of the integral Harnack type inequalities for sign-changing solutions.



\end{abstract}

  \pagestyle{myheadings} \thispagestyle{plain}
\markboth{I. Skrypnik}
{On the H\"{o}lder continuity...}

\section{Introduction and main results}\label{Introduction}
 In this paper we are concerned with doubly nonlinear parabolic equations
 with measurable coefficients
 \begin{equation}\label{eq1.1}
 \frac{\partial}{\partial t}\big(|u|^{q-1} u\big)- div \mathbf{A}(x, t, u, D u)=0,\quad (x, t)\in \Omega_T:=\Omega\times (0, T),\quad \Omega\subset \mathbb{R}^N.
 \end{equation}
Throughout the paper we suppose that the functions $\mathbf{A} :\Omega_T\times \mathbb{R}^{N+1}\rightarrow  \mathbb{R}^N$
are such that $\mathbf{A}(\cdot, \cdot, u, \xi)$ are Lebesgue measurable for all $u\in \mathbb{R}$, $\xi\in \mathbb{R}^N$, and
 $\mathbf{A}(x, t, \cdot, \cdot)$ are continuous for almost all $(x,t)\in \Omega_T.$  We also assume that the following structure conditions are satisfied
 \begin{equation}\label{eq1.2}
 \begin{cases}
 \mathbf{A}(x, t, u, \xi)\,\xi\geqslant K_1\,|\xi|^p,\quad \xi\in \mathbb{R}^N,\\
 |\mathbf{A}(x, t, u, \xi)|\leqslant K_2\,|\xi|^{p-1},
 \end{cases}
 \end{equation}
where $K_1$, $K_2$ are positive constants. The numbers $K_1$, $K_2$, $p$, $q$, $N$ are further
 referred to as the data and  we will write $\gamma$ as a generic positive constant that can be quantitatively determined
 a priori only in terms of the data and that can change from line to line.

Postponing the formal definitions of weak solution, we will proceed to present the main results. 
\begin{theorem}\label{th1.1}
Let u be a locally bounded, local, weak solution to \eqref{eq1.1}, \eqref{eq1.2} in $\Omega_T$, assume also that 
\begin{equation}\label{eq1.3}
0<q<p-1,\quad p<2,\quad p+N(p-2)>0,
\end{equation}
or
\begin{equation}\label{eq1.4}
p-1<q, \quad p>2, \quad q p+N(p-1-q)>0.
\end{equation}
Then $u$ is locally H\"{o}lder continuous in $\Omega_T$.
\end{theorem}

Our main result, Theorem \ref{th1.1} is a consequence of the following statement. Fix $(x_0, t_0)\in \Omega_T$, construct the cylinder
$$Q_{r, b\, \omega^{q-p+1} r^p}:=B_r(x_0)\times (t_0- b\,\omega^{q-p+1} r^p, t_0)\subset \Omega_T$$ 
and let $\mu^+$, $\mu^-$ and $\omega$ be the numbers such that
\begin{equation}\label{eq1.5}
\mu^+\geqslant \sup\limits_{Q_{r, b\,\omega^{q-p+1} r^p}} u,\quad \mu^-\leqslant \inf\limits_{Q_{r, b\, \omega^{q-p+1} r^p}} u,\quad \omega \geqslant \mu^+ -\mu^-,
\end{equation}
where $b\in (0, 1)$ to be fixed depending only on the data.

\begin{propo}\label{pr1.1}
Let u be a locally bounded, local, weak solution to \eqref{eq1.1}, \eqref{eq1.2} in $\Omega_T$, let  conditions \eqref{eq1.3} or \eqref{eq1.4}  be fulfilled.
There exist numbers $b$, $\sigma_0$, $\eta_0$, $\epsilon_0\in (0, 1)$  depending only on the data such that if
\begin{equation}\label{eq1.6}
\max\big(\mu^+, -\mu^-\big)\leqslant (1+\eta_0)\, \omega,
\end{equation}
then
\begin{equation}\label{eq1.7}
\osc\limits_{Q_{\epsilon_0 r, b(\sigma_0 \omega)^{q-p+1} (\epsilon_0 r)^p}} u \leqslant \sigma_0\,\omega.
\end{equation}
\end{propo}

A mathematical interest of  equation \eqref{eq1.1} lies in its singularity or/and degeneracy at points where either $u = 0$ or $Du = 0$. Moreover, it covers the porous medium equation ($p = 2$), the parabolic $p$-Laplace equation ($q = 1$) and
 Trudinger's equation ($q = p -1$).  These classes of equations have numerous applications
 and have been attracting attention for several decades (see, see, e.g. the monographs \cite{DiB1, DiB3, Lad, Vaz, Wu }, survey \cite{DiB2}
 and references therein).  
 
 The issue of local H\"{o}lder regularity for  equation \eqref{eq1.1} has been investigated by a number of authors, in various forms and with different notions of solution. 
 The well-studied ranges are the borderline case $p-1=q$, the doubly singular case, i.e.
$p-1<q$ and $1<p<2$, as well as the doubly degenerate case, i.e. $q < p-1$ and $p > 2$, see \cite{And, Ara, Bog1, Bog2, Bog3, Bog4, Cas, Cia2, Die, Gia, Hen, Hen1, Hen2, Iva1, Iva2, Iva3, Iva4, Kin1, Kuu1, Kuu2, Lia, Lia1, Por, Urb, Ves1, Ves} for the details. It seems, the degenerate/singular case $q< p-1$ and $p<2$, as well as the singular/ degenerate case $p-1< q$ and $p>2$ for sign changing solutions  have not yet been studied. These cases were investigated only for non-negative solutions in \cite{Cia1}. The main novelty of our results lies in the extension of the known ranges of $q$ and $p$ to other, previously unstudied cases, namely, to degenerate/ singular and singular/ degenerate ranges.

The main step in the proof of Proposition \ref{pr1.1}  are the integral Harnack type inequalities for sign-changing solutions. To formulate these results, we note the following simple facts
$$\pm\big(|\mu^{\pm}|^{q-1}\mu^{\pm}-|u|^{q-1}u\big)\geqslant 0,\quad \text{if}\quad \mu^-\leqslant u\leqslant \mu^+.$$
\begin{theorem}\label{th1.2}
Let $u$ be a  locally bounded, local, weak solution to \eqref{eq1.1}, \eqref{eq1.2} and assume that $q<p-1$ and $p< 2$. Construct the cylinder
$Q_{r, t-s}(y, t):=B_{r}(y)\times (s, t) \subset \Omega_T$, and let $\mu^-\leqslant u \leqslant \mu^+$ in $Q_{r, t-s}(y, t)$. Then for any 
$\delta,\,\sigma \in (0, 1)$, any $l\geqslant \frac{2 p^2}{2-p}$, any $\zeta(x)\in C^1_0(B_{r}(y))$, such that $\zeta(x)=1$ in $B_{r(1-\sigma)}(y)$, $0\leqslant \zeta(x)\leqslant 1$, $|D \zeta(x)|\leqslant \frac{1}{\sigma r}$ there holds
\begin{multline}\label{eq1.8}
\sup\limits_{s\leqslant \tau\leqslant  t}\fint\limits_{B_{r}(y)\times\{\tau\}}\pm\big(|\mu^{\pm}|^{q-1}\mu^{\pm}-|u|^{q-1}u\big)\zeta^l(x)\,dx \leqslant \\\leqslant \frac{1}{1-\delta} \inf\limits_{s\leqslant\tau \leqslant t}\fint\limits_{B_{r}(y)\times\{\tau\}}\pm\big(|\mu^{\pm}|^{q-1}\mu^{\pm}-|u|^{q-1}u\big)\zeta^l(x)\,dx+\\+\frac{\gamma}
{(1-\delta)\delta^\gamma\,\sigma^\gamma}
\big[\max(\mu^+, -\mu^-)\big]^{\frac{(1-q)(p-1)}{2-p}}
\Big(\frac{t-s}{r^p}\Big)^{\frac{1}{2-p}},
\end{multline}
with a constant $\gamma >0$ depending only on the data and $l$.
\end{theorem}
\begin{theorem}\label{th1.3}
Let $u$ be a  locally bounded, local, weak solution to \eqref{eq1.1}, \eqref{eq1.2} and assume that $q> p-1$ and $q>1$. Construct the cylinder
$Q_{r, t-s}(y, t):=B_{r}(y)\times (s, t) \subset \Omega_T$, let $\mu^-\leqslant u \leqslant \mu^+$ in $Q_{r, t-s}(y, t)$.  
Then for any 
$\delta,\,\sigma \in (0, 1)$, any $l\geqslant \frac{2 q p^2}{q-p+1}$, any $\zeta(x)\in C^1_0(B_{r}(y))$, such that $\zeta(x)=1$ in $B_{r(1-\sigma)}(y)$, $0\leqslant \zeta(x)\leqslant 1$, $|D \zeta(x)|\leqslant \frac{1}{\sigma\,r}$ there holds
\begin{multline}\label{eq1.9}
\sup\limits_{s\leqslant \tau\leqslant t}\fint\limits_{B_{r}(y)\times\{\tau\}}\pm\big(|\mu^{\pm}|^{q-1}\mu^{\pm}-|u|^{q-1}u\big)\zeta^l(x)\, dx \leqslant\\\leqslant \frac{1}{1-\delta}
\inf\limits_{s\leqslant\tau\leqslant t}\fint\limits_{B_{r}(y)\times\{\tau\}} \pm\big(|\mu^{\pm}|^{q-1}\mu^{\pm}-|u|^{q-1}u\big)\zeta^l(x)\,dx +\frac{\gamma}{(1-\delta)\delta^\gamma \sigma^\gamma}
\Big(\frac{t-s}{r^{p}}\Big)^{\frac{q}{q-p+1}},
\end{multline}
with a constant $\gamma >0$ depending only on the data and $l$.
\end{theorem}

\begin{remark}
Theorems \ref{th1.2}, \ref{th1.3} are well-known for the parabolic $p$-Laplacian  or  for non-negative solutions of porous medium equation, as well as of doubly nonlinear equations with $\pm(\mu^{\pm}-u)$ replaced by $u$ (see e.g.\cite{DiB1, DiB3, Bog3}).
Inequality similar to \eqref{eq1.8} was proved in \cite{Cia1} for non-negative solutions to doubly nonlinear parabolic equations.
\end{remark}

In the proof, we use DiBenedetto's rescaling method \cite{DiB1}.
The main idea is to reduce the proof of the H\"{o}lder continuity to a singular case, that is, either a singular parabolic $p$-Laplacian in the case $q < p -1$ and $p < 2$, or a singular parabolic equation with double nonlinearity in the case $q > p -1$. Thus, we  work in cylinders $Q_{r, b\,\omega^{q-p+1} r^p}$ with some small  $b\in (0, 1)$.   Using the integral Harnack inequalities, we fix number $b\in (0, 1)$ depending only on the data and obtain a result on the propagation of positivity in measure for all $t\in (t_0-b\,\omega^{q-p+1} r^p, t_0)$.  This result together with the expansion of positivity, Lemma \ref{lem2.6} (see Section $2$ below) implies inequality \eqref{eq1.7} in the degenerate/ singular case. So, having  Theorem \ref{th1.2} in hand, the proof of Proposition \ref{pr1.1} in the case $q<p-1$ and $p<2$ becomes almost  obvious (see Section $3$ below). In the  case  $q>p-1$  we  additionally need the following theorem on the expansion of positivity for sign  solutions. 

\begin{theorem}\label{th1.4}
Let $u$ be a locally bounded, local, weak solution to \eqref{eq1.1}, \eqref{eq1.2} in $\Omega_T$ and let $q> p-1$. Assume also that
\begin{equation}\label{eq1.10}
|B_r(x_0)\cap\big\{\pm (\mu^{\pm}-u(\cdot, t))\geqslant \xi\,\omega\big\}|\geqslant \alpha|B_r(x_0)|,
\end{equation}
for all $t\in (t_0-b\,\omega^{q-p+1} r^p, t_0)$ and with some $b$, $\xi$, $\alpha \in (0, 1)$, then there exist numbers   $\epsilon_*$, $\xi_*\in (0, 1)$ depending only on the data, $\xi$, $\alpha$ and $b$ such that
\begin{equation}\label{eq1.11}
\mp u(x, t)\geqslant \xi_*\,\omega,\quad  (x, t)\in B_{\frac{r}{2}}(x_0)\times (t_0-\frac{1}{4}\,b\,\omega^{q-p+1} r^p, t_0),
\end{equation}
provided that 
\begin{equation}\label{eq1.12}
\pm\mu^{\pm}\leqslant \epsilon_*\,\omega.
\end{equation}
\end{theorem}
A similar result for non-negative super-solutions of   singular parabolic equations has been established in \cite{DiB3, Bog3, Bog4} as a key tool to study Harnack's inequality (in this case $\mu^-=0$ and condition $(1.12)_-$ is always satisfied). Using it to handle the H\"{o}lder regularity seems new in the doubly nonlinear and singular/ degenerate setting. 

 The rest of the paper contains the proof of the above theorems. In Section $2$ we collect some auxiliary
 propositions and required integral and point-wise estimates of solutions.  In Sections $3$, $4$ we give a proof of Proposition \ref{pr1.1}
 which is based on the integral Harnack type inequalities for sign-changing solutions. Section $5$ contains the proof of H\"{o}lder continuity. Various forms of the integral  Harnack type inequality, Theorems \ref{th1.2} and \ref{th1.3} are proved in Appendix A.
Finally, expansion of positivity for sign solutions, Theorem \ref{th1.4} we prove in Appendix B.

\section{ Preliminaries}

\subsection{Notations}
First, we will provide some notations that will be used in the sequel, more precisely,
for fixed $r, \eta >0$ we define the following  cylinders 
$$Q_{r, \eta}(y, s):=B_r(y)\times (s-\eta, s).$$
Integral averages are marked as usual
$$\fint\limits_E u\,dx:=\frac{1}{|E|}\int\limits_E u\,dx.$$
For any $a \in \mathbb{R}$  we also use the notation $a_{\pm}:= \max(\pm a, 0)$.  

\subsection{Notion of Solution}

A function
\begin{equation*}
u\in C_{loc}(0, T; L^{q+1}_{loc}(\Omega))\cap L^p_{loc}(0, T; W^{1, p}_{loc}(\Omega))
\end{equation*}
 is a local, weak sub(super)-solution to \eqref{eq1.1}, \eqref{eq1.2}, if for every
 compact set $E\subset \Omega$ and every sub-interval $[t_1,  t_2] \subset (0, T]$
\begin{equation}\label{eq2.1}
\int\limits_{E}|u|^{q-1} u\,\zeta\,dx\Big|^{t_2}_{t_1}+\int\limits_{t_1}^{t_2}\int\limits_E\big\{-|u|^{q-1} u \zeta_t+ \mathbf{A}(x, \tau, u, D u)\big\}\,dx\,d\tau \leqslant (\geqslant ) 0,
\end{equation}
 for all non-negative test functions
 $$\zeta \in W^{1, q+1}_{loc}(0, T; L^{q+1}(E))\cap L^p_{loc}(0, T; W^{1, p}_0(E)).$$
A function $u$  that is both is a local weak sub-solution and a local weak super-solution to \eqref{eq1.1}, \eqref{eq1.2} is  a local weak solution.

\subsection{Mollification in Time}
The time derivative of a weak solution exists in the sense of distribution only.   In 
order to overcome the lack of regularity in the time variable, we define the following 
mollification in time:
$$\llbracket v \rrbracket_h(x, t):=\frac{1}{h}\int\limits_0^t e^{\frac{s-t}{h}}\,v(x, s)\,ds,\quad v\in L^1(\Omega_T).$$
We refer the reader to \cite{Kin} for various properties of this mollification.

\subsection{Algebraic Lemmas}
The following lemma can be extracted from \cite[Lemma 2.1]{Bog1}
\begin{lemma}\label{lem2.1}
There exists a constant $c = c(q)>0$ such that, for all $q>0$, 
 $a, b\in \mathbb{R}$, the following inequality holds true:
 \begin{equation*}
 \frac{1}{c}\,\big(|a|+|b|\big)^{q-1}|a-b|\leqslant ||a|^{q-1}a -|b|^{q-1} b|\leqslant c\,\big(|a|+|b|\big)^{q-1}|a-b|.
 \end{equation*}
\end{lemma}
In what follows we will use the following evident inequalities
\begin{lemma}\label{lem2.2}
For any $\epsilon$, $\sigma \in (0, 1)$, $\bar{\sigma}\in (0, \epsilon)$, $0<p-1<q$  there hold
\begin{equation}\label{eq2.2}
\int\limits^\sigma_0\frac{z^{q-1}\,dz}{[z+\epsilon]^{p-1}}=\sigma^q\,\int\limits^1_0\frac{z^{q-1}\,dz}{[\sigma\,z+\epsilon]^{p-1}}
\geqslant \sigma^q\,\int\limits^1_0\frac{z^{q-1}\,dz}{[z+\epsilon]^{p-1}},
\end{equation}
\begin{equation}\label{eq2.3}
\int\limits^\epsilon_0\frac{z^{q-1}\,dz}{[z+\epsilon]^{p-1}}=\epsilon^{q-p+1}\,\int\limits^1_0\frac{z^{q-1}\,dz}{[z+1]^{p-1}}
\leqslant \epsilon^{q-p+1}\,\int\limits^1_0\frac{z^{q-1}\,dz}{[z+\epsilon]^{p-1}},
\end{equation}
\begin{equation}\label{eq2.4}
\int\limits^{\bar{\sigma}}_0\frac{z^{q-1}\,dz}{[\epsilon-z]^{p-1}}\leqslant \bar{\sigma}^q\Big(\frac{1}{\epsilon-\bar{\sigma}}\Big)^{p-1}\int\limits
^1_0 z^{q-1}\,dz\leqslant
\bar{\sigma}^q\Big(\frac{1+\epsilon}{\epsilon-\bar{\sigma}}\Big)^{p-1}\int\limits
^1_0\frac{z^{q-1}\,dz}{[z+\epsilon]^{p-1}}.
\end{equation}
\end{lemma}

\subsection{Local energy estimates}
The following lemma can be found, for example  in \cite{Bog1}

\begin{lemma}\label{lem2.3}
Let $u$ be a local weak sub (super)-solution to \eqref{eq1.1}, \eqref{eq1.2} in $\Omega_T$. Then there exists $\gamma>0$ depending only on the data such that for every cylinder $Q_{r, \theta}(y, s) \subset Q_{8r, 8\theta}(y, s)\subset \Omega_T$, every
$k\in \mathbb{R}$ and any smooth $\zeta(x, t)$ which is zero for $(x, t) \in \partial B_r(y)\times (s-\theta, s)$ one has
\begin{multline}\label{eq2.5}
\sup\limits_{s-\theta\leqslant t\leqslant s}\int\limits_{B_r(y)\times\{t\}}g_{\pm}(u, k)\,\zeta^p\,dx+
\iint\limits_{Q_{r,\theta}(y, s)}|D(u-k)_{\pm}|^p\,\zeta^p\,dx dt\leqslant\\\leqslant \int\limits_{B_r(y)\times\{s-\theta\}}g_{\pm}(u, k)\,\zeta^p\,dx+\gamma\iint\limits_{Q_{r, \theta}(y, s)}\Big[g_{\pm}(u, k)\zeta^{p-1}\,|\zeta_t|+(u-k)_{\pm}^p\,|D \zeta|^p\Big]\,dx\,dt.
\end{multline}
Here
$$g_{\pm}(u, k):=q\,\int\limits_0^{(u-k)_{\pm}}|z \pm k|^{q-1}\,z\,dz.$$
\end{lemma}

\subsection{A De Giorgi Type Lemma}
Construct the cylinder $Q_{r, \theta}(y, s) \subset Q_{8r, 8\theta}(y, s)\subset \Omega_T$ and let $\mu^+$, $\mu^-$, $\omega$ be the numbers such that
$$\mu^+\geqslant \sup\limits_{Q_{r, \theta}(y, s)} u,\quad \mu^-\leqslant \inf\limits_{Q_{r, \theta}(y, s)} u,\quad \mu^+-\mu^-\leqslant \omega.$$
 The proof of the following lemma is almost standard and uses Lemma \ref{lem2.3} and Sobolev embedding theorem (see, e. g. \cite[Lemma 6.1]{Lia1}).
 
\begin{lemma}\label{lem2.4}
Let $u$ be a locally bounded, local, weak sub (super)-solution to \eqref{eq1.1}, \eqref{eq1.2} in $\Omega_T$ and let $q>0$, $p>1$.
Fix $\xi\in (0, 1)$, there exists $\nu\in (0, 1)$ depending only on the data and $\xi$, $\omega$,  $r$, $\theta$ such that if
\begin{equation}\label{eq2.6}
\big|Q_{r, \theta}(y, s)\cap \big\{\pm(\mu^{\pm}-u)\leqslant \xi\,\omega\big\}\big|\leqslant \nu |Q_{r, \theta}(y, s)|,
\end{equation}
then either
\begin{equation}\label{eq2.7}
|\mu^{\pm}| >8\,\xi\,\omega
\end{equation}
or
\begin{equation}\label{eq2.8}
\pm(\mu^{\pm}-u)\geqslant \frac{1}{2}\,\xi\,\omega,\quad (x, t)\in Q_{\frac{1}{2}r, \frac{1}{2}\theta}(y, s).
\end{equation}
\end{lemma}

The proof of the next lemma is completely similar to that of Lemma \ref{lem2.4} and  we give a short sketch here.

\begin{lemma}\label{lem2.5}
Let $u$ be a locally bounded, local, weak sub (super)-solution to \eqref{eq1.1}, \eqref{eq1.2} in $\Omega_T$ and let $q>0$, $p>1$.
Fix $\xi\in (0, 1)$, there exists $\nu\in (0, 1)$ depending only on the data and $\xi$, $\omega$,  $r$, $\theta$ such that if
\begin{equation}\label{eq2.9}
\big|Q_{r, \theta}(y, s)\cap \big\{\mp u\leqslant \xi\,\omega\big\}\big|\leqslant \nu |Q_{r, \theta}(y, s)|,
\end{equation}
then either
\begin{equation}\label{eq2.10}
\pm \mu^{\pm} >8\,\xi\,\omega
\end{equation}
or
\begin{equation}\label{eq2.11}
\mp u\geqslant \frac{1}{2}\,\xi\,\omega,\quad (x, t)\in Q_{\frac{1}{2}r, \frac{1}{2}\theta}(y, s).
\end{equation}
\end{lemma}
\begin{proof}
Fix $\sigma \in (0, 1)$, let $\rho$, $\eta$ be the numbers such that $\frac{1}{2} r\leqslant \rho(1-\sigma)\leqslant \rho \leqslant r$,
$\frac{1}{2}\theta\leqslant \eta(1-\sigma) \leqslant \eta \leqslant \theta$ and $\zeta_1(x)\in C^1_0(B_\rho(y))$, $\zeta_1(x)=1$ in $B_{\rho(1-\sigma)}(y)$, $0\leqslant \zeta_1(x)\leqslant 1$, $|D \zeta_1(x)|\leqslant \frac{1}{\sigma\,\rho}$, $\zeta_2(t)\in C^1(\mathbb{R})$,
$0\leqslant \zeta_2(t)\leqslant 1$, $\zeta_2(t)=0$ for $t\leqslant s-\eta$, $\zeta_2(t)=1$ for $t\geqslant s-(1-\sigma)\eta$, $|\zeta'_2(t)|
\leqslant \dfrac{1}{\sigma\eta}$ and let $\zeta(x, t)=\zeta_1(x) \zeta_2(t)$. Use inequality \eqref{eq2.5} of Lemma \ref{lem2.3} for the function $(u\pm k)_{\pm}$, with
$\frac{1}{2}\xi\,\omega\leqslant k(1-\sigma)\leqslant k\leqslant  \xi\omega$, we obtain
\begin{multline*}
\sup\limits_{s-\eta\leqslant t\leqslant s}\int\limits_{B_\rho(y)\times\{t\}}\tilde{g}_{\pm}(u, k)\,\zeta^p\,dx+
\iint\limits_{Q_{\rho,\eta}(y, s)}|D(u\pm k)_{\pm}|^p\,\zeta^p\,dx dt\leqslant\\\leqslant \gamma\iint\limits_{Q_{\rho, \eta}(y, s)}\Big[\tilde{g}_{\pm}(u, k)\zeta^{p-1}\,|\zeta_t|+(u\pm k)_{\pm}^p\,|D \zeta|^p\Big]\,dx\,dt,
\end{multline*}
where $\tilde{g}_{\pm}(u, k)=q\,\int\limits^{(u\pm k)_{\pm}}_0 |z - k|^{q-1}\,z\,dz$.  Further we assume that \eqref{eq2.10} is violated, i.e. $$\pm\mu^{\pm}\leqslant 8 \xi \omega.$$
Then 
$$(u\pm k)_{\pm}\leqslant \pm \mu^{\pm}+k\leqslant \gamma \xi\,\omega,$$
and
\begin{equation*}
\tilde{g}_{\pm}(u, k)=q\int\limits^{\pm u}_{-k}|z|^{q-1}(z+k) dz \leqslant q\,(8\xi \omega + k)\,\int\limits^{8\xi \omega}_{- k}|z|^{q-1}\,dz=(8\xi \omega + k)((8 \xi\omega)^q + k^q)\leqslant \gamma (\xi \omega)^{q+1}.
\end{equation*}
Using the algebraic lemma, Lemma \ref{lem2.1} and the facts that
$$\frac{1}{2}\xi\omega \leqslant k\leqslant \pm u\leqslant |u|+\frac{1}{2}|\pm u-k|\leqslant \pm \frac{3}{2} u\leqslant \pm  \frac{3}{2} \mu^{\pm}\leqslant 12 \xi\omega, \quad\text{if}\quad \pm u\geqslant k,$$
and
$$\frac{1}{4} \xi \omega \leqslant \frac{1}{2} k\leqslant |u|+\frac{1}{2}(k\mp u)\leqslant |u|+\frac{1}{2}|\pm u-k|\leqslant 2 k\leqslant 2 \xi \omega,\quad \text{if} \quad -k\leqslant \pm u\leqslant k,$$
we obtain  for $\pm u \geqslant -k$
\begin{multline*}
\tilde{g}_{\pm}(u, k)=q\int\limits^{\pm u}_{-k}|z|^{q-1}(z+k) dz\geqslant q\int\limits^{\pm u}_{\frac{1}{2}(\pm u-k)}|z|^{q-1}(z+k)\,dz
\geqslant \frac{q}{2}(\pm u +k)\,\int\limits^{\pm u}_{\frac{1}{2}(\pm u-k)}|z|^{q-1}\,dz=\\=
\frac{1}{2}(\pm u +k)\Big[\pm|u|^{q-1} u- \Big(\frac{1}{2}\Big)^q |\pm u-k|^{q-1}(\pm u-k)\Big]\geqslant\\\geqslant \frac{1}{4 \gamma}(\pm u +k)^2
\Big[|u|+\frac{1}{2} |\pm u-k|\Big]^{q-1}\geqslant \frac{1}{\gamma}(\xi \omega)^{q-1}(\pm u +k)^2=\frac{1}{\gamma}(\xi \omega)^{q-1}
(u\pm k)^2_{\pm}.
\end{multline*}
So, collecting the last three inequalities we arrive at
\begin{multline*}
(\xi \omega)^{q-1}\sup\limits_{s-\eta\leqslant t\leqslant s}\,\int\limits_{B_\rho(y)\times\{t\}}(u \pm k)^2_{\pm}\,\zeta^p\,dx+
\iint\limits_{Q_{\rho,\eta}(y, s)}|D(u\pm k)_{\pm}|^p\,\zeta^p\,dx dt\leqslant\\\leqslant \frac{\gamma}{(\sigma \rho)^p}(\xi \omega)^p\Big[1+\frac{(\xi \omega)^{q-p+1} \rho^p}{\eta}\Big]|Q_{\rho, \eta}(y, s)\cap\{\pm u \geqslant -k\}|.
\end{multline*}
From this, by the H\"{o}lder inequality and Sobolev embedding theorem we obtain
\begin{multline*}
(\sigma k)^p\,|Q_{\rho(1-\sigma), \eta(1-\sigma)}(y, s)\cap\{\pm u \geqslant -k (1-\sigma)\}|\leqslant\iint\limits_{Q_{\rho, \eta}(y, s)}(u\pm k)^p_{\pm}\,\zeta^{p^2}\,dx dt\leqslant\\\leqslant \gamma \Big(\sup\limits_{s-\eta\leqslant t\leqslant s}\,\int\limits_{B_\rho(y)\times\{t\}}(u \pm k)^2_{\pm}\,\zeta^{2 p}\,dx\Big)^{\frac{p}{N+2}}\Big(\iint\limits_{Q_{\rho,\eta}(y, s)}\big|D\big[(u\pm k)_{\pm}\zeta^p\big]\big|^p\,dx dt\Big)^{\frac{N}{N+2}}\times\\\times|Q_{\rho, \eta}(y, s)\cap\{\pm u \geqslant -k\}|^{\frac{2}{N+2}}\leqslant\\\leqslant\frac{\gamma}{\sigma^\gamma}(\xi \omega)^{\frac{(1-q)p}{N+2}}\bigg(\Big(\frac{\xi \omega}{\rho}\Big)^p\Big[1+\frac{(\xi \omega)^{q-p+1} \rho^p}{\eta}\Big]\bigg)^{\frac{N+p}{N+2}}
|Q_{\rho, \eta}(y, s)\cap\{\pm u \geqslant -k\}|^{1+\frac{p}{N+2}},
\end{multline*}
which yields
\begin{multline*}
\frac{|Q_{\rho(1-\sigma), \eta(1-\sigma)}(y, s)\cap\{\pm u \geqslant -k (1-\sigma)\}|}{|Q_{\rho, \eta}(y, s)|}\leqslant\\\leqslant \frac{\gamma}{\sigma^\gamma}\Big(\frac{\eta}{(\xi \omega)^{q-p+1}\rho^p}\Big)^{\frac{p}{N+2}}\Big[1+\frac{(\xi \omega)^{q-p+1} \rho^p}{\eta}\Big]^{\frac{N+p}{N+2}} \bigg(\frac{|Q_{\rho, \eta}(y, s)\cap\{\pm u \geqslant -k\}|}{|Q_{\rho, \eta}(y, s)|}
\bigg)^{1+\frac{p}{N+2}},
\end{multline*}
from this by iteration the required \eqref{eq2.11} follows, provided that
\begin{equation}\label{eq2.12}
\nu\leqslant \frac{1}{\gamma}\,\frac{(\xi \omega)^{q-p+1} r^p}{\theta}\Big[1+\frac{(\xi \omega)^{q-p+1} r^p}{\theta}\Big]^{-\frac{N+p}{p}},
\end{equation}
which completes the proof of the lemma.
\end{proof}
\begin{remark}
As it was mentioned, the number $\nu$ in \eqref{eq2.9}, as well as in \eqref{eq2.6} can be chosen to satisfy \eqref{eq2.12}.
Further we will use Lemmas \ref{lem2.4}, \ref{lem2.5} with $\theta=b\,(\xi\,\omega)^{q-p+1}\,r^p$ and some   $b\in (0, 1)$. Simple calculations give that the number
$$\nu=\frac{1}{\gamma}\,b^{\frac{N}{p}}$$
satisfies  condition \eqref{eq2.12}.
\end{remark}

\subsection{Expansion of Positivity}
The following two lemmas can be found, for example in \cite[Lemma 3.4]{Lia} and \cite[Lemma 3.3]{Bog2}

\begin{lemma}\label{lem2.6}
Let $u$ be a locally bounded, local, weak sub(super)-solution to \eqref{eq1.1}, \eqref{eq1.2} in $\Omega_T$, $q>0$, $1<p<2$ and suppose also 
that
\begin{equation}\label{eq2.13}
\lambda\,\omega \leqslant \pm \mu^{\pm}\leqslant \Lambda\,\omega,
\end{equation}
with some $0<\lambda <\Lambda$ and
\begin{equation}\label{eq2.14}
\big|B_r(y)\cap\big\{\pm(\mu^{\pm}-u(\cdot, s))\geqslant a\,\omega\big\}\big|\geqslant \alpha |B_r(y)|,
\end{equation}
with some $a<\frac{1}{2}\,\lambda$ and some $\alpha \in (0, 1)$. Then there exist constants $\eta$, $\delta\in (0, 1)$ depending
only on the data,  $\alpha$, $\lambda$, $\Lambda$ such that
\begin{equation}\label{eq2.15}
\pm(\mu^{\pm}-u)\geqslant \eta\,a\,\omega \quad \text{in}\quad  B_{2r}(y)\times (s+\,\delta\,a^{2-p}\,\omega^{q-p+1}\,r^p),
\end{equation}
provided that $B_{2r}(y)\times (s, s+ \delta\,a^{2-p}\,\omega^{q-p+1}\,r^p)\subset \Omega_T$.
 \end{lemma}
 
 \begin{lemma}\label{lem2.7}
 Let $u$ be a locally bounded, local, weak sub(super)-solution to \eqref{eq1.1}, \eqref{eq1.2} in $\Omega_T$, $q>0$, $p>2$ and suppose also 
that
\begin{equation}\label{eq2.16}
\lambda\,\omega \leqslant \pm \mu^{\pm}\leqslant \Lambda\,\omega,
\end{equation}
with some $0<\lambda <\Lambda$ and
\begin{equation}\label{eq2.17}
\big|B_r(y)\cap\big\{\pm(\mu^{\pm}-u(\cdot, s))\geqslant a\,\omega\big\}\big|\geqslant \alpha |B_r(y)|,
\end{equation}
with some $a<\frac{1}{2}\,\lambda$ and some $\alpha \in (0, 1)$. Then there exist constants $B>1$, $\eta \in (0, 1)$ depending
only on the data,  $\alpha$, $\lambda$, $\Lambda$ such that
\begin{equation}\label{eq2.18}
\pm(\mu^{\pm}-u)\geqslant \eta\,a\,\omega \quad \text{in}\quad  B_{2r}(y)\times (s+\frac{1}{2} B\,\omega^{q-p+1}\,r^p, s+\,B\,\omega^{q-p+1}\,r^p),
\end{equation}
provided that $B_{2r}(y)\times (s, s+ B\,\omega^{q-p+1}\,r^p)\subset \Omega_T$.
 \end{lemma}

 \subsection{Notion of Parabolicity}
The following lemma will be used in the sequel, it can be extracted, for example from \cite{Bog2}.
\begin{lemma}\label{lem2.8}
Let $u$ be a local weak sub(super)-solution to \eqref{eq1.1}, \eqref{eq1.2}. Then, for any
$k\in \mathbb{R}$, the truncation $k\pm (u-k)_{\pm}$ is a local weak sub(super)-solution to \eqref{eq1.1}, \eqref{eq1.2}.
 \end{lemma}

\section{Proof of Proposition \ref{pr1.1} Under Conditions $q<p-1$ and $p<2$}

In the case $q<p-1$ and $p<2$ we assume that $\eta_0=1$ and hence, we assume that the following inequality holds
\begin{equation}\label{eq3.1}
\max\big(\mu^+, -\mu^-\big)\leqslant 2\,\omega.
\end{equation} 

In what follows  we can suppose that
\begin{equation}\label{eq3.2}
\max(\mu^+, -\mu^-) \geqslant \frac{1}{4}\omega,
\end{equation}
since otherwise 
\begin{equation}\label{eq3.3}
\osc\limits_{Q_{r, b\, \omega^{q-p+1} r^p}} u\leqslant \frac{1}{2}\,\omega,
\end{equation}
here $b\in (0, 1)$ is the number to be fixed later.
Condition \eqref{eq3.2} implies that either
\begin{equation}\label{eq3.4}
\mu^+=\max (\mu^+, -\mu^-) \geqslant \frac{1}{4}\,\omega,
\end{equation}
or
\begin{equation}\label{eq3.5}
-\mu^-=\max (\mu^+, -\mu^-) \geqslant \frac{1}{4}\,\omega.
\end{equation}

\subsection{Proof of Proposition \ref{pr1.1} Under Condition \eqref{eq3.4}}
Under conditions \eqref{eq3.4} we  assume that one of the following alternatives holds:
either
\begin{equation}\label{eq3.6}
|Q_{\frac{3}{4}r, b\,\omega^{q-p+1} (\frac{3}{4}r)^p}\cap\{u\leqslant \mu^-+\frac{1}{4}\,\omega\}|\leqslant \nu\,|Q_{\frac{3}{4}r, b\,\omega^{q-p+1} (\frac{3}{4}r)^p}|,
\end{equation}
or
\begin{multline}\label{eq3.7}
|Q_{\frac{3}{4}r, b\,\omega^{q-p+1} (\frac{3}{4}r)^p}\cap\{u\leqslant \mu^+-\frac{1}{4}\,\omega\}|\geqslant
|Q_{\frac{3}{4}r, b\,\omega^{q-p+1} (\frac{3}{4} r)^p}\cap\{u\leqslant \mu^-+\frac{1}{4}\,\omega\}|\geqslant\\\geqslant \nu\,|Q_{\frac{3}{4}r, b\,\omega^{q-p+1} (\frac{3}{4} r)^p}|,
\end{multline}
where $\nu=\frac{1}{\gamma}\,b^{\frac{N}{p}}$,  is the number defined in Remark $2.1$.

\subsubsection{Analysis of the First Alternative}
Let inequality \eqref{eq3.6} holds , using the evident fact that $\max(\mu^+, -\mu^-)=\max(|\mu^+|, |\mu^-|)$, by \eqref{eq3.1} and  a De Giorgi type lemma, Lemma \ref{lem2.4} we obtain
$$\inf\limits_{Q_{\frac{1}{2} r, \frac{1}{2} b\,\omega^{q-p+1} r^p}} u \geqslant \mu^-+\frac{1}{8}\,\omega,$$
which yields
\begin{equation}\label{eq3.8}
\osc\limits_{Q_{\frac{1}{2}r, \frac{1}{2} b\,\omega^{q-p+1} r^p}} u\leqslant \mu^+-\mu^--\frac{1}{8}\,\omega\leqslant \frac{7}{8}\,\omega.
\end{equation}

\subsubsection{Analysis of the Second Alternative}
Under condition \eqref{eq3.7} there exists a time level $s\in (t_0-b\,\omega^{q-p+1}(\frac{3}{4}r)^p, t_0)$ such that
\begin{equation*}
\big|B_{\frac{3}{4} r}(x_0)\cap \big\{u(\cdot, s)\leqslant \mu^+-\frac{1}{4}\omega\big\}\big|\geqslant \nu\,|B_{\frac{3}{4}r}(x_0)|,\quad \nu=\frac{1}{\gamma}\,b^{\frac{N}{p}}.
\end{equation*}

By our choices $q<1$, we use Theorem \ref{th1.2} with  $\delta=\frac{1}{2}$, $\sigma=\frac{1}{4}$ and $l=\frac{2p^2}{2-p}$ for the function $|\mu^+|^{q-1}\mu^+-|u|^{q-1}u$ in the cylinder $Q_{r, b\, \omega^{q-p+1} r^p}$. Note that 
$$\mu^+-u\leqslant |\mu^+|+|u|\leqslant 2 \max(\mu^+, -\mu^-),$$
and then, by the algebraic lemma, Lemma \ref{lem2.1}
$$\frac{1}{\gamma}[\max(\mu^+, -\mu^-)]^{q-1}(\mu^+-u)\leqslant |\mu^+|^{q-1}\mu^+-|u|^{q-1} u\leqslant \gamma (\mu^+-u)^q.$$
So, using \eqref{eq3.1} we obtain for all $t\in (t_0-b\,\omega^{q-p+1}\,r^p, t_0)$ 
\begin{multline}\label{eq3.9}
\frac{1}{\gamma}\,b^{\frac{N}{p}}\,\omega^q=\frac{1}{\gamma}\,\nu\,\omega^q\leqslant \frac{1}{\gamma} \big[\max(\mu^+, -\mu^-)\big]^{q-1}\fint\limits_{B_{\frac{3}{4}r}(x_0)\cap\{u(\cdot, s)\leqslant \mu^+-\frac{1}{4}\omega\}}\big(\mu^+-u(x, s)\big)\,dx\leqslant\\\leqslant 
\gamma\,\fint\limits_{B_{\frac{3}{4} r}(x_0)\times\{s\}}\Big[|\mu^+|^{q-1}\mu^+-|u|^{q-1} u\Big]\,dx\leqslant
\gamma\,\fint\limits_{B_r(x_0)\times\{t\}}\Big[|\mu^+|^{q-1}\mu^+-|u|^{q-1} u\Big]\,dx
+\\+\gamma\,\big[\max(\mu^+, -\mu^-)\big]^{\frac{(1-q)(p-1)}{2-p}}(b\omega^{q-p+1})^{\frac{1}{2-p}}\leqslant
\gamma \fint\limits_{B_r(x_0)\times\{t\}}(\mu^+-u)^q\,dx+\gamma b^{\frac{1}{2-p}}\,\omega^q\leqslant\\\leqslant
\gamma \Big[\epsilon^q+b^{\frac{1}{2-p}}\Big]\omega^q+ \gamma\,\omega^q\frac{|B_r(x_0)\cap\{\mu^+-u(\cdot, t)\geqslant \epsilon \omega\}|}{|B_r(x_0)|}.
\end{multline}
Fix $b\in (0, 1)$ from the condition 
$\gamma\,b^{\frac{1}{2-p}}\leqslant \frac{1}{16\gamma}\,b^{\frac{N}{p}}$, that is $b^{\frac{p+N(p-2)}{p(2-p)}}\leqslant \frac{1}{16\gamma}$
and then choosing $\epsilon=\epsilon(b) \in (0, 1)$ such that $\gamma \epsilon^q=\frac{1}{16 \gamma}\,b^{\frac{N}{p}},$
from \eqref{eq3.9} we arrive at
\begin{equation}\label{eq3.10}
|B_r(x_0)\cap\{\mu^+-u(\cdot, t)\geqslant \epsilon \omega\}|\geqslant \alpha |B_r(x_0)|,\quad \alpha=
\frac{1}{\gamma}b^{\frac{N}{p}}\in (0, 1),
\end{equation}
for all $t\in (t_0-b\,\omega^{q-p+1}\,r^p, t_0)$. Now we use Lemma \ref{lem2.6}, by \eqref{eq3.1}, \eqref{eq3.4} and \eqref{eq3.10} this lemma is applicable, so, there exist $\eta_1$, $\delta_1 \in (0, 1)$ depending only on the data
\begin{equation*}
u\leqslant \mu^+-\eta_1\,\epsilon\,\omega\quad \text{in}\quad B_{\frac{1}{2}r}(x_0)\times (t_0-(b-\delta_1\,\epsilon^{2-p})\omega^{q-p+1} r^p, t_0),
\end{equation*}
choosing $\epsilon$ smaller if necessary, this  inequality yields
\begin{equation}\label{eq3.11}
\osc\limits_{Q_{\frac{1}{2}r, \frac{1}{2} b\,\omega^{q-p+1} r^p}} u\leqslant (1-\eta_1\,\epsilon)\,\omega.
\end{equation}

\subsection{Proof of Proposition \ref{pr1.1} under Condition \eqref{eq3.5}}

Under condition \eqref{eq3.5} we change the alternatives and  assume that  
either
\begin{equation}\label{eq3.12}
|Q_{\frac{3}{4}r, b\,\omega^{q-p+1} (\frac{3}{4}r)^p}\cap\{u\geqslant \mu^+-\frac{1}{4}\,\omega\}|\leqslant \nu\,|Q_{\frac{3}{4}r, b\,\omega^{q-p+1} (\frac{3}{4}r)^p}|,
\end{equation}
or
\begin{multline}\label{eq3.13}
|Q_{\frac{3}{4}r, b\,\omega^{q-p+1} (\frac{3}{4}r)^p}\cap\{u\geqslant \mu^- +\frac{1}{4}\,\omega\}|\geqslant
|Q_{\frac{3}{4}r, b\,\omega^{q-p+1} (\frac{3}{4} r)^p}\cap\{u\geqslant \mu^+-\frac{1}{4}\,\omega\}|\geqslant\\\geqslant \nu\,|Q_{\frac{3}{4}r, b\,\omega^{q-p+1} (\frac{3}{4} r)^p}|,
\end{multline}
where $\nu=\frac{1}{\gamma}\,b^{\frac{N}{p}}$ is the number defined in Remark $2.1$.

The proof is completely similar to the previous one, so, we give a sketch here.
Under condition \eqref{eq3.12} we use a De Giorgi type lemma, Lemma \ref{lem2.4} and obtain
\begin{equation}\label{eq3.14}
\osc\limits_{Q_{\frac{1}{2} r, \frac{1}{2} b\, \omega^{q-p+1}\,r^p}} u\leqslant \frac{7}{8}\,\omega.
\end{equation}
If condition \eqref{eq3.13} holds, then there exists a time level $s\in (t_0-b\,\omega^{q-p+1}(\frac{3}{4}r)^p, t_0)$ such that
\begin{equation*}
\big|B_{\frac{3}{4} r}(x_0)\cap \big\{u(\cdot, s)\geqslant \mu^-+\frac{1}{4}\omega\big\}\big|\geqslant \nu\,|B_{\frac{3}{4}r}(x_0)|,\quad \nu=\frac{1}{\gamma}\,b^{\frac{N}{p}}.
\end{equation*}
We use Theorem \ref{th1.2} for the function $|u|^{q-1}u- |\mu^-|^{q-1}\mu^-$ in the cylinder $Q_{r, b\,\omega^{q-p+1}r^p}$. Similarly  to \eqref{eq3.10},
choosing $b$ and $\epsilon(b)$ small enough, we obtain
\begin{equation}\label{eq3.15}
|B_r(x_0)\cap\{u(\cdot, t)-\mu^-\geqslant \epsilon \omega\}|\geqslant \alpha |B_r(x_0)|,\quad \alpha=
\frac{1}{\gamma}\,b^{\frac{N}{p}}\in (0, 1),
\end{equation}
for all $t\in (t_0-b\,\omega^{q-p+1}\,r^p, t_0)$. Now we use Lemma \ref{lem2.6}, by \eqref{eq3.1}, \eqref{eq3.5} and \eqref{eq3.15} this lemma is applicable, so, there exist $\eta_2$, $\delta_2 \in (0, 1)$ depending only on the data
\begin{equation*}
u\geqslant \mu^-+\eta_2\,\epsilon\,\omega\quad \text{in}\quad B_{\frac{1}{2}r}(x_0)\times (t_0-(b-\delta_2\,\epsilon^{2-p})\omega^{q-p+1} r^p, t_0),
\end{equation*}
choosing $\epsilon$ smaller if necessary, this  inequality yields
\begin{equation}\label{eq3.16}
\osc\limits_{Q_{\frac{1}{2}r, \frac{1}{2} b\,\omega^{q-p+1} r^p}} u\leqslant (1-\eta_2\,\epsilon)\,\omega.
\end{equation}

\subsubsection{Proof of Proposition \ref{pr1.1} Concluded}

Collecting \eqref{eq3.3}, \eqref{eq3.8}, \eqref{eq3.11}, \eqref{eq3.14}, \eqref{eq3.16} and choosing $\sigma_0=\max(1-\eta_1\epsilon, 1-\eta_2 \epsilon)$ and $\epsilon_0=\frac{1}{2}\,\sigma^{\frac{p-q-1}{p}}_0$ we arrive at
\begin{equation}\label{eq3.17}
\osc\limits_{Q_{\epsilon_0 r, b\,(\sigma_0 \omega)^{q-p+1} (\epsilon_0 r)^p}} u \leqslant \sigma_0\,\omega,
\end{equation}
which  completes the proof of Proposition \ref{pr1.1} under conditions $q<p-1$ and $p<2$.

\section{Proof of Proposition \ref{pr1.1} Under Conditions $q>p-1$ and $p>2$}
Recall that we prove Proposition \ref{pr1.1} under the condition
\begin{equation}\label{eq4.1}
\max(\mu^+, -\mu^-)\leqslant (1+\eta_0)\,\omega.
\end{equation} 
Fix  $\epsilon_*\in (0, 1)$ to be defined depending only on the data. In what follows we assume that  either 
$$\min(\mu^+, -\mu^-)\leqslant \epsilon_*\,\omega,$$
or
$$\min(\mu^+, -\mu^-)\geqslant \epsilon_*\,\omega.$$
And hence, the following three
different cases are possible:  either
\begin{equation}\label{eq4.2}
-\mu^-=\min(\mu^+, -\mu^-)\leqslant \epsilon_*\,\omega ,\quad \text{and}\quad \mu^+=\max(\mu^+, -\mu^-)\leqslant (1+\eta_0)\,\omega,
\end{equation}
or
\begin{equation}\label{eq4.3}
\mu^+=\min(\mu^+, -\mu^-)\leqslant \epsilon_*\, \omega,\quad \text{and}\quad -\mu^-=\max(\mu^+, -\mu^-)\leqslant (1+\eta_0)\,\omega,
\end{equation}
or
\begin{equation}\label{eq4.4}
\epsilon_*\,\omega\leqslant \min(\mu^+, -\mu^-)\leqslant \max(\mu^+, -\mu^-)\leqslant (1+\eta_0)\,\omega.
\end{equation}

\subsection{Proof of Proposition \ref{pr1.1} Under Condition \eqref{eq4.2}}

Under condition \eqref{eq4.2} we assume that one of the following alternative cases holds: either
\begin{equation}\label{eq4.5}
\big|Q_{\frac{3}{4} r, b\omega^{q-p+1} r^p}\cap\big\{u\geqslant \mu^+-\frac{1}{4}\omega\big\}\big|\leqslant \nu|Q_{\frac{3}{4}r, b\omega^{q-p+1}r^p}|,
\end{equation}
or
\begin{multline}\label{eq4.6}
\big|Q_{\frac{3}{4} r, b\omega^{q-p+1} r^p}\cap\big\{u\geqslant \mu^-+\frac{1}{4}\omega\big\}\big|\geqslant \big|Q_{\frac{3}{4} r, b\omega^{q-p+1} r^p}\cap\big\{u\geqslant \mu^+-\frac{1}{4}\omega\big\}\big|\geqslant\\\geqslant \nu|Q_{\frac{3}{4}r, b\omega^{q-p+1}r^p}|,
\end{multline}
where  $\nu=\frac{1}{\gamma}\,b^{\frac{N}{p}}$ is the number defined in Remark $2.1$.

\subsubsection{Analysis of the First Alternative}
If condition \eqref{eq4.5} holds we use a De Giorgi type lemma, Lemma \ref{lem2.4}, by the second inequality in \eqref{eq4.2} this lemma is applicable, hence we obtain
$$u\leqslant \mu^+-\frac{1}{8}\,\omega,\quad \text{in}\quad Q_{\frac{1}{2}r, \frac{1}{2} b\omega^{q-p+1} r^p},$$
which yields
\begin{equation}\label{eq4.7}
\osc\limits_{Q_{\frac{1}{2}r, \frac{1}{2} b\omega^{q-p+1} r^p}} u\leqslant \frac{7}{8}\,\omega.
\end{equation}

\subsubsection{Analysis of the Second Alternative}
Under condition \eqref{eq4.6} there exists a time level $s\in (t_0-b\omega(\frac{3}{4} r)^p, t_0)$ such that
\begin{equation*}
\big|B_{\frac{3}{4} r}(x_0)\cap\big\{u(\cdot, s)\geqslant \mu^-+\frac{1}{4}\omega\big\}\big|\geqslant \nu\,|B_{\frac{3}{4}r}(x_0)|.
\end{equation*}
By the fact that $q>1$, using the algebraic lemma, Lemma \ref{lem2.1}
\begin{multline*}
[\pm(\mu^{\pm}-u)]^q\leqslant \gamma \big[\pm(\mu^{\pm}-u)\big]\,(|\mu^{\pm}|+|u|)^{q-1}\leqslant \gamma \big[\pm(|\mu^{\pm}|^{q-1}\mu^{\pm}-|u|^{q-1}u)\big]\leqslant\\\leqslant \gamma \big[\pm(\mu^{\pm}-u)\big]\,(|\mu^{\pm}|+|u|)^{q-1}
\leqslant \gamma [\max(\mu^+, -\mu^-)]^{q-1} \pm(\mu^{\pm}-u).
\end{multline*}
Now we use Theorem \ref{th1.3} with $\delta=\frac{1}{2}$, $\sigma=\frac{1}{4}$ and $l=\frac{2 q p^2}{q-p+1}$ for the function $|u|^{q-1}u-|\mu^-|^{q-1}\mu^-$ in the cylinder $Q_{r, b\,\omega^{q-p+1} r^p}$. Using \eqref{eq4.2} we obtain for all $t\in (t_0-b \omega^{q-p+1} r^p, t_0)$
\begin{multline*}
\frac{1}{\gamma}\,b^{\frac{N}{p}}\,\omega^q\leqslant \frac{1}{\gamma}\fint\limits_{B_{\frac{3}{4} r}(x_0)\cap\{u(\cdot, s)\geqslant \mu^-+\frac{1}{4}\omega\}}(u(x, s)-\mu^-)^q\,dx\leqslant\\\leqslant 
\gamma \fint\limits_{B_{\frac{3}{4} r}(x_0)\times\{s\}}\Big[|u|^{q-1}u-|\mu^-|^{q-1}\mu^-\Big]\,dx\leqslant\\\leqslant
\gamma
\fint\limits_{B_r(x_0)\times\{t\}}\Big[|u|^{q-1}u-|\mu^-|^{q-1}\mu^-\Big]\,dx +\gamma \omega^q\,b^{\frac{q}{q-p+1}}\leqslant\\\leqslant
\gamma\,\omega^{q-1}\fint\limits_{B_r(x_0)\times\{t\}}(u-\mu^-)\,dx+\gamma \omega^q\,b^{\frac{q}{q-p+1}}\leqslant\\\leqslant
\gamma \Big[\epsilon+b^{\frac{q}{q-p+1}}\Big]\,\omega^q+\gamma \omega^q\frac{|B_r(x_0)\cap\{u(\cdot, t)-\mu^-\geqslant \epsilon \omega\}|}{|B_r(x_0)|},
\end{multline*}
choosing $b$ from the condition $\gamma b^{\frac{q}{q-p+1}}=\frac{1}{4\gamma}b^{\frac{N}{p}}$, that is $b^{\frac{qp+N(p-q-1)}{p(q-p+1)}}=\frac{1}{4\gamma}$, and then assuming that  $\epsilon\in (0, 1)$ is so small that $\gamma\,\epsilon\leqslant \frac{1}{4\gamma}b^{\frac{N}{p}}$, from the previous we arrive at
\begin{equation}\label{eq4.8}
|B_r(x_0)\cap\{u(\cdot, t)-\mu^-\geqslant \epsilon \omega\}|\geqslant \alpha |B_r(x_0)|,\quad \alpha=\frac{1}{2\gamma} b^{\frac{N}{p}},
\end{equation}
for all $t\in(t_0-b \omega^{q-p+1} r^p, t_0)$.  Theorem \ref{th1.4} ensures the existence
of $\epsilon_*$, $\xi^{(1)}_*\in (0, 1)$ depending only on the data such that
$$u \geqslant \xi^{(1)}_*\,\omega,\quad \text{in}\quad Q_{\frac{1}{2} r, \frac{1}{4} b \omega^{q-p+1} r^p},$$
provided that the first inequality in \eqref{eq4.2} holds. And hence, by the second inequality in \eqref{eq4.2}
\begin{equation}\label{eq4.9}
\osc\limits_{Q_{\frac{1}{2} r, \frac{1}{4} b \omega^{q-p+1} r^p}} u\leqslant \mu^+ -\xi^{(1)}_*\,\omega \leqslant (1+\eta_0-\xi^{(1)}_*)\,\omega \leqslant (1-\frac{1}{2}\,\xi^{(1)}_*)\,\omega,
\end{equation}
provided that $\eta_0\leqslant \frac{1}{2}\,\xi^{(1)}_*.$

\subsection{Proof of Proposition \ref{pr1.1} Under Condition \ref{eq4.3}}

Under condition \eqref{eq4.3} the proof  is completely similar to the previous one, only we need to change the alternatives: either
\begin{equation}\label{eq4.10}
\big|Q_{\frac{3}{4} r, b\omega^{q-p+1} r^p}\cap\big\{u\leqslant \mu^-+\frac{1}{4}\omega\big\}\big|\leqslant \nu|Q_{\frac{3}{4}r, b\omega^{q-p+1}r^p}|,
\end{equation}
or
\begin{multline}\label{eq4.11}
\big|Q_{\frac{3}{4} r, b\omega^{q-p+1} r^p}\cap\big\{u\leqslant \mu^+-\frac{1}{4}\omega\big\}\big|\geqslant \big|Q_{\frac{3}{4} r, b\omega^{q-p+1} r^p}\cap\big\{u\leqslant \mu^-+\frac{1}{4}\omega\big\}\big|\geqslant\\\geqslant \nu|Q_{\frac{3}{4}r, b\omega^{q-p+1}r^p}|,
\end{multline}
where  $\nu=\frac{1}{\gamma}\,b^{\frac{N}{p}}$ is the number defined in Remark $2.1$.
Using a De Giorgi type lemma, Lemma \ref{lem2.4}, from \eqref{eq4.10} we obtain
\begin{equation}\label{eq4.12}
\osc\limits_{Q_{\frac{1}{2} r, \frac{1}{2} b\omega^{q-p+1}r^p}} u\leqslant \frac{7}{8}\,\omega.
\end{equation}
Under condition \eqref{eq4.11} there holds 
$$\big|B_{\frac{3}{4} r}(x_0)\cap\big\{u(\cdot, s)\leqslant \mu^+-\frac{1}{4} \omega\big\}\big|\geqslant \nu |B_{\frac{3}{4} r}(x_0)|,$$
with some $s \in (t_0-b\omega^{q-p+1}(\frac{3}{4} r)^p, t_0).$ Completely similar to \eqref{eq4.8}, using Theorem \ref{th1.3} for the function
$|\mu^+|^{q-1}\mu^+-|u|^{q-1}u$ in the cylinder $Q_{r, b\,\omega^{q-p+1} r^p}$ and choosing $b$, $\epsilon$ small enough we obtain for all $t\in (t_0-b\omega^{q-p+1} r^p, t_0)$
\begin{equation}\label{eq4.13}
\big|B_r(x_0)\cap\big\{\mu^+-u(\cdot, t)\geqslant \epsilon\,\omega\big\}\big|\geqslant \alpha |B_r(x_0)|,\quad \alpha=\frac{1}{\gamma} b^{\frac{N}{p}}.
\end{equation}
Choose $\epsilon_*\in (0, 1)$ as in Theorem \ref{th1.4}, by the first inequality in \eqref{eq4.3} this theorem ensures the existence
of $\xi^{(2)}_*\in (0, 1)$ depending only on the data such that
$$-u \geqslant \xi^{(2)}_*\,\omega,\quad \text{in}\quad Q_{\frac{1}{2} r, \frac{1}{4} b \omega^{q-p+1} r^p}.$$
And hence, by the second inequality in \eqref{eq4.3}
\begin{equation}\label{eq4.14}
\osc\limits_{Q_{\frac{1}{2} r, \frac{1}{4} b \omega^{q-p+1} r^p}} u \leqslant -\mu^--\xi^{(2)}_*\,\omega\leqslant (1+\eta_0-\xi^{(2)}_*)\omega \leqslant (1-\frac{1}{2}\xi^{(2)}_*)\,\omega,
\end{equation}
provided that $\eta_0\leqslant \frac{1}{2}\,\xi^{(2)}_*.$

\subsection{Proof of Propostion \ref{pr1.1} Under Condition \eqref{eq4.4}}

Condition \eqref{eq4.4} guarantees that both $\mu^+$ and $\mu^-$ satisfy
\begin{equation}\label{eq4.15}
\epsilon_*\,\omega \leqslant \pm \mu^{\pm}\leqslant (1+\eta_0)\,\omega\leqslant 2\,\omega,
\end{equation}
so, we can use the expansion of positivity, Lemma \ref{lem2.7}. For this define $\bar{\epsilon}=\big(\frac{b}{B}\big)^{\frac{1}{p}}$,
where $B>1$ is the number depending only on the data to be defined and construct the cylinder
$$Q_{\bar{r}, B\omega^{q-p+1} \bar{r}^p}:=B_{\bar{r}}(x_0)\times(t_0-B\omega^{q-p+1}\bar{r}^p, t_0)\subset Q_{r, b\,\omega^{q-p+1} r^p},
\quad \bar{r}:=\bar{\epsilon}\,r.$$
We assume that one of the alternative cases holds: either
\begin{equation}\label{eq4.16}
\big|B_{\bar{r}}(x_0)\cap\big\{\mu^+-u(\cdot, t_0-B\,\omega^{q-p+1}\bar{r}^p)\geqslant \frac{1}{2}\epsilon_*\,\omega\big\}\big|\geqslant \frac{1}{2}|B_{\bar{r}}(x_0)|,
\end{equation}
or
\begin{multline}\label{eq4.17}
\big|B_{\bar{r}}(x_0)\cap\big\{u(\cdot, t_0-B\,\omega^{q-p+1}\bar{r}^p)-\mu^-\geqslant \frac{1}{2}\epsilon_*\,\omega\big\}\big|\geqslant\\
\geqslant
\big|B_{\bar{r}}(x_0)\cap\big\{\mu^+-u(\cdot, t_0-B\,\omega^{q-p+1}\bar{r}^p)\leqslant \frac{1}{2}\epsilon_*\,\omega\big\}\big|\geqslant \frac{1}{2}|B_{\bar{r}}(x_0)|.
\end{multline}
Let us assume, for example \eqref{eq4.16}, the case of \eqref{eq4.17} is analogous. The number $\epsilon_*$ was defined  depending only on the data, so by \eqref{eq4.15}, \eqref{eq4.16}, Lemma \ref{lem2.7} ensures the existence of $B>1$ and $\eta_1 \in (0, 1)$ depending only on the data  such that
\begin{equation*}
u\leqslant \mu^+-\eta_1\,\omega,\quad \text{in}\quad Q_{\frac{1}{2}\bar{r}, \frac{1}{2}B \omega^{q-p+1} \,\bar{r}^p},
\end{equation*}
that is 
\begin{equation}\label{eq4.18}
\osc\limits_{Q_{\frac{1}{2}\bar{r}, \frac{1}{2}b \omega^{q-p+1} \,\bar{r}^p}} u\leqslant
\osc\limits_{Q_{\frac{1}{2}\bar{r}, \frac{1}{2}B \omega^{q-p+1} \,\bar{r}^p}} u\leqslant (1-\eta_1)\,\omega.
\end{equation}

\subsection{Proof of Proposition \ref{pr1.1} Concluded}
Collecting \eqref{eq4.7}, \eqref{eq4.9},  \eqref{eq4.14}, \eqref{eq4.18}  and choosing $\epsilon_0:=\frac{1}{4}\bar{\epsilon} $, $\sigma_0:=\max(1-\frac{1}{2}\xi^{(1)}_*, 1-\frac{1}{2}\xi^{(2)}_*, 1-\eta_1)$, $\eta_0:=\frac{1}{2}\,\min (\xi^{(1)}_*, \xi^{(2)}_*)$,  we arrive at
\begin{equation*}
\osc\limits_{Q_{\epsilon_0 r, b\, (\sigma_0 \omega)^{q-p+1} \,(\epsilon_0 r)^p}} u\leqslant \sigma_0\,\omega,
\end{equation*}
which completes the proof of Proposition \ref{pr1.1}.

\section{H\"{o}lder Continuity, Proof of Theorem \ref{th1.1}}

 The proof in both cases $q<p-1$ and $p<2$, as well as $q>p-1$ and $p>2$ is almost the same and is based on Propositon \ref{pr1.1}, so we present it simultaneously.
 
Fix $(x_0, t_0)\in \Omega_T$, let $M:=\sup\limits_{\Omega_T}|u|$ and construct the cylinder $Q_{2R}\subset \Omega_T$
$$
Q_{R}:=\begin{cases}
B_R(x_0)\times (t_0-R^{q+1}, t_0),\quad \text{if}\quad q<p-1,\quad p<2,\\
B_R(x_0)\times (t_0-(2M)^{q-p+1} R^p, t_0),\quad \text{if}\quad q>p-1,\quad p>2.
\end{cases}
$$
If for all $\rho \in (0, R]$
$$\osc\limits_{Q_\rho} u\leqslant \,\rho,$$
 then there is nothing to prove. Assume that there exists $\rho_0\in (0, R]$ such that
\begin{equation}\label{eq5.1}
\osc\limits_{Q_{\rho_0}} u\geqslant \,\rho_0.
\end{equation}
In this case set $\mu^+_0:=\sup\limits_{Q_{\rho_0}} u$, $\mu^-_0:=\inf\limits_{Q_{\rho_0}} u$ and $\omega_0:=2\max(\mu^+_0, -\mu^-_0)$,
by \eqref{eq5.1} 
$$\omega_0\geqslant \osc\limits_{Q_{\rho_0}} u\geqslant \rho_0,$$
and hence
$$Q_{\rho_0, b\,\omega^{q-p+1}_0\rho^p_0}=B_{\rho_0}(x_0)\times (t_0-b\,\omega^{q-p+1}_0 \rho^p_0, t_0)\subset Q_{\rho_0},$$
with $b\in (0, 1)$ to be defined.
By our choices
\begin{equation}\label{eq5.2}
\max(\mu^+_0, -\mu^-_0)=\frac{1}{2}\,\omega_0\,< (1+\eta_0)\,\omega_0,
\end{equation}
where $\eta_0\in (0, 1]$ to be defined.

\subsection{Reduction to Oscillation Near Zero}

Proposition \ref{pr1.1} ensures the existence of $b$, $\epsilon_0$, $\eta_0$, $\sigma_0 \in (0, 1)$ such that
\begin{equation}\label{eq5.3}
\osc\limits_{Q_{r_1, b\,\omega^{q-p+1}_1 r^p_1}} u\leqslant \omega_1,\quad r_1=\epsilon_0\,\rho_0,\quad \omega_1=\sigma_0\,\omega_0.
\end{equation}
We can repeat the previous arguments assuming
on each step that 
\begin{equation}\label{eq5.4}
\max\big(\mu^+_{j-1}, -\mu^-_{j-1}\big)\leqslant (1+\eta_0)\,\omega_{j-1}
\end{equation}
and obtain
\begin{equation}\label{eq5.5}
\osc\limits_{Q_{r_j, b\,\omega^{q-p+1}_j r^p_j}} u\leqslant \omega_j,\quad r_j=\epsilon^j_0\,\rho_0,\quad \omega_j=\sigma^j_0\,\omega_0.
\end{equation}

\subsection{Reduction to Oscillation Away From Zero}

Assume now that $j_0\geqslant 1$ is the first number such that 
\begin{equation}\label{eq5.6}
\begin{cases}
\osc\limits_{Q_{r_{j_0}, b\,\omega^{q-p+1}_{j_0} r^p_{j_0}}} u\leqslant \omega_{j_0},\\
\text{and}\\
\max\big(\mu^+_{j_0}, -\mu^-_{j_0}\big)\geqslant (1+\eta_0)\,\omega_{j_0}.
\end{cases}
\end{equation}
Inequalities \eqref{eq5.6} yield
$$
\begin{cases}
\mu^-_{j_0}\geqslant \frac{\eta_0}{1+\eta_0}\mu^+_{j_0}\geqslant \eta_0\,\omega_{j_0},\quad \text{if}\quad \mu^+_{j_0}=\max(\mu^+_{j_0}, -\mu^-_{j_0}),\\
-\mu^+_{j_0}\geqslant -\frac{\eta_0}{1+\eta_0} \mu^-_{j_0}\geqslant \eta_0\,\omega_{j_0},\quad \text{if}\quad -\mu^-_{j_0}=\max(\mu^+_{j_0}, -\mu^-_{j_0}).
\end{cases}
$$
By the fact that the second inequality in \eqref{eq5.6} is violated for $j=j_0-1$ (for $j=0$ this inequality is violated by \eqref{eq5.2}), we have
\begin{multline}\label{eq5.7}
\eta_0\,\omega_{j_0}\leqslant \frac{\eta_0}{1+\eta_0}\,\mu^+_{j_0}\leqslant \mu^-_{j_0}\leqslant \mu^+_{j_0-1}\leqslant 
\max(\mu^+_{j_0-1}, -\mu^-_{j_0-1})\leqslant\\\leqslant (1+\eta_0) \omega_{j_0-1}=\frac{1+\eta_0}{\sigma_0}\omega_{j_0},\qquad \text{if}\quad
\mu^+_{j_0}=\max(\mu^+_{j_0}, -\mu^-_{j_0}).
\end{multline}
Similarly,
\begin{multline}\label{eq5.8}
\eta_0\,\omega_{j_0}\leqslant -\frac{\eta_0}{1+\eta_0}\mu^-_{j_0}\leqslant -\mu^+_{j_0}\leqslant -\mu^-_{j_0}\leqslant -\mu^-_{j_0-1}\leqslant \max(\mu^+_{j_0-1}, -\mu^-_{j_0-1})\leqslant\\\leqslant (1+\eta_0)\, \omega_{j_0-1}=\frac{1+\eta_0}{\sigma_0}\omega_{j_0},\qquad \text{if}\quad
-\mu^-_{j_0}=\max(\mu^+_{j_0}, -\mu^-_{j_0}).
\end{multline}
Let us consider the case $\mu^+_{j_0}=\max(\mu^+_{j_0}, -\mu^-_{j_0})$, the  case $-\mu^-_{j_0}=\max(\mu^+_{j_0}, -\mu^-_{j_0})$ is analogous. By \eqref{eq5.7}  $u$ is a non-negative 
solution to \eqref{eq1.1}, \eqref{eq1.2} in the cylinder $Q_{r_{j_0}, b\omega^{q-p+1}_{j_0} r^p_{j_0}}$ such that
\begin{equation}\label{eq5.9}
\eta_0\,\omega_{j_0}\leqslant \mu^-_{j_0}\leqslant u\leqslant\mu^+_{j_0}\leqslant \frac{(1+\eta_0)^2}{\eta_0 \sigma_0}\,\omega_{j_0} \quad \text{in}\quad Q_{r_{j_0}, b\omega^{q-p+1}_{j_0} r^p_{j_0}}.
\end{equation}
Introduce the change of variables and new unknown function (for the details we refer the reader, for example to \cite[Section 4.5]{Bog2} if $p>2$ and \cite[Section 4.5]{Lia} if $p<2$ )
$$\tau=\frac{t}{\omega^{q-p+1}_{j_0}},\qquad v=\Big(\frac{u}{\omega_{j_0}}\Big)^q,$$
then $v$ is a non-negative solution to the  parabolic $p$-Laplacian ($p>2$  or $p<2$)
\begin{equation}\label{eq5.10}
\begin{cases}
v_\tau-div\tilde{\mathbf{A}}(x, \tau, v, D v)=0,\quad (x, \tau)\in Q_{r_{j_0}, b\, r^p_{j_0}},\\
\tilde{\mathbf{A}}(x, \tau, v, D v)\,D v\geqslant \tilde{K}_1\,|D v|^p,\,\,\,\tilde{K}_1=\tilde{K}_1(K_1, \sigma_0, \eta_0)>0,\\
|\tilde{\mathbf{A}}(x, \tau, v, D v)|\leqslant \tilde{K}_2\,|D v|^{p-1},\,\,\,\tilde{K}_2=\tilde{K}_2(K_2, \sigma_0, \eta_0)>0,\\
\eta_0^q\leqslant v\leqslant \frac{(1+\eta_0)^{2q}}{(\eta_0 \sigma_0)^q}.
\end{cases}
\end{equation}
Moreover, using the first inequality in \eqref{eq5.6}, by the last inequality in \eqref{eq5.10} and by the evident fact that 
\begin{equation*}
a^q-b^q\leqslant
\begin{cases}
(a-b)^q,\quad a\geqslant b\geqslant 0,\quad 0<q\leqslant1,\\
q(a-b)(a^{q-1}+b^{q-1}),\quad a\geqslant b\geqslant 0,\quad q>1,
\end{cases}
\end{equation*}
we obtain
\begin{equation*}
\osc\limits_{Q_{r_{j_0}, r^p_{j_0}}} v\leqslant \frac{q+1}{\omega^q_{j_0}}\max\Big[\omega^q_{j_0}, \omega_{j_0}\,\Big([\mu^-_{j_0}]^{q-1}+[\mu^+_{j_0}]^{q-1}\Big) \Big]\leqslant \gamma_0(\sigma_0, \eta_0):=\tilde{\omega}_{j_0}.
\end{equation*}
Choosing
$$
\tilde{r}_{j_0}:=
\min\big(1, \frac{1}{[\gamma_0(\sigma_0, \eta_0)]^{\frac{2-p}{p}}}\big)\,r_{j_0},
$$
from the previous we obtain
\begin{equation}\label{eq5.11}
\osc\limits_{Q_{\tilde{r}_{j_0}, b\,\tilde{\omega}^{2-p}_{j_0} \tilde{r}^p_{j_0}}} v\leqslant \osc\limits_{Q_{r_{j_0}, r^p_{j_0}}} v\leqslant \tilde{\omega}_{j_0}.
\end{equation}
By \cite[Chapters 3, 4]{DiB1} inequality \eqref{eq5.11} ensures the existence of $\tilde{\epsilon}$, $\tilde{\sigma}\in (0, 1)$
such that 
\begin{equation}\label{eq5.12}
\begin{cases}
\tilde{r}_{j_0+j}:=\tilde{\epsilon}^j \tilde{r}_{j_0},\quad \tilde{\omega}_{j_0+j}:=\tilde{\sigma}^j\,\tilde{\omega}_{j_0},\\
\osc\limits_{Q_{\tilde{r}_{j_0+j}, b\,\tilde{\omega}^{2-p}_{j_0+j} \tilde{r}^p_{j_0+j}}} v\leqslant \tilde{\omega}_{j_0+j},\qquad j\geqslant 0.
\end{cases}
\end{equation}
Returning to the original variables, redefining $r_{j_0+j}=\tilde{\epsilon}^j_0 r_{j_0}$ with $\tilde{\epsilon}_0:=\tilde{\epsilon}_0(\sigma_0, \eta_0, \tilde{\epsilon})\in (0, 1)$ and choosing $\omega_{j_0+j}:=\gamma_0(\sigma_0, \eta_0)\,\tilde{\sigma}^j\,\omega_{j_0+j}$, from \eqref{eq5.12} we get
\begin{equation}\label{eq5.13}
\osc\limits_{Q_{r_{j_0+j}, b\omega^{q-p+1}_{j_0+j} r^p_{j_0+j}}} u\leqslant \omega_{j_0+j},\qquad j\geqslant 0.
\end{equation}

\subsection{The Final Argument}
To complete the proof of the H\"{o}lder continuity we collect estimates \eqref{eq5.5}, \eqref{eq5.13}. Redefining $r_j$, $\omega_j$
we arrive at 
\begin{equation}\label{eq5.14}
\begin{cases}
\epsilon:=\epsilon_0 \tilde{\epsilon}_0 \min\big(1, \frac{1}{[\gamma_0(\sigma_0, \eta_0)]^{\frac{q-p+1}{p}}}\big),\quad \sigma:=\max(\sigma_0, \tilde{\sigma})\in (0, 1),\quad \omega'_0:=\gamma_0(\sigma_0, \eta_0)\omega_0,\\
r_j:=\epsilon^j\,\rho_0,\quad\omega_j:=\sigma^j\,\omega'_0,\qquad
\osc\limits_{Q_{r_j, b \,\omega_j^{q-p+1} r^p_j}} u\leqslant \omega_j,\quad j\geqslant 0.
\end{cases}
\end{equation}
The H\"{o}lder continuity of $u$ follows from \eqref{eq5.14}
by standard iterative arguments, see, for example \cite[Chapter 3]{DiB1}, this completes the proof of
Theorem \ref{th1.1}.

\section{Appendix A. Proof of Theorems \ref{th1.2}, \ref{th1.3}}

The beginning of the proof of Theorems \ref{th1.2} and \ref{th1.3} is almost the same, so,  consider the cutoff function $\zeta(x)$, defined 
in Theorems \ref{th1.2}, \ref{th1.3}. Assume without loss that $s=0$, test identity \eqref{eq2.1} by $\zeta^l(x)$ with $l>0$ large enough,
and obtain for any $t_1$, $t_2\in (0, t)$
\begin{multline}\label{eq6.1}
\int\limits_{B_{r}(y)\times\{t_1\}}\pm\big(|\mu^{\pm}|^{q-1}\mu^{\pm}-|u|^{q-1}u\big)\zeta^l(x)\,dx\leqslant
q\int\limits_{B_{r}(y)\times\{t_1\}}\int\limits^{\pm(\mu^{\pm}-u)}_0\,|z\mp \mu^{\pm}|^{q-1}\,dz\,\zeta^l(x)\,dx\leqslant\\\leqslant
q\int\limits_{B_{r}(y)\times\{t_2\}}\int\limits^{\pm(\mu^{\pm}-u)}_0\,|z\mp \mu^{\pm}|^{q-1}\,dz\,\zeta^l(x)\,dx+
\frac{\gamma l}{\sigma r}\int\limits^t_0\int\limits_{B_{r}(y)}|D u|^{p-1} \zeta^{l-1}(x)\,dx\,d\tau\leqslant\\\leqslant\int\limits_{B_{r}(y)\times\{t_2\}}\pm\big(|\mu^{\pm}|^{q-1}\mu^{\pm}-|u|^{q-1}u\big)\zeta^l(x)\,dx+
\frac{\gamma l}{\sigma r}\int\limits^t_0\int\limits_{B_{r}(y)}|D u|^{p-1} \zeta^{l-1}(x)\,dx\,d\tau.
\end{multline}

\subsection{Proof of Theorem \ref{th1.2}, Case $q<p-1$ and $p<2$}

Let us  estimate the second integral on the right-hand side of \eqref{eq6.1}, for this note that by our choices $q<1$.
\begin{lemma}\label{lem6.1}
For any $\delta\in (0, 1)$ there holds
\begin{multline}\label{eq6.2}
\frac{\gamma l}{\sigma r}\int\limits^t_0\fint\limits_{B_{r}(y)}|D u|^{p-1} \zeta^{l-1}(x)\,dx\,d\tau\leqslant \\\leqslant \delta\,\sup\limits_{0\leqslant\tau\leqslant t}\mathcal{I}(\tau)+\frac{\gamma(l)}{\delta^\gamma \sigma^\gamma}\big[\max(\mu^+, -\mu^-)\big]^{\frac{(1-q)(p-1)}{2-p}}
\Big(\frac{t}{r^p}\Big)^{\frac{1}{2-p}},
\end{multline}
where
$$\mathcal{I}(\tau):=\fint\limits_{B_{r}(y)\times\{\tau\}}
\pm\big(|\mu^{\pm}|^{q-1}\mu^{\pm}-|u|^{q-1}u\big)\zeta^l(x)\,dx.$$
\end{lemma}
\begin{proof}
Set $\mathcal{M}:=\max(\mu^+, -\mu^-)$, $\epsilon:=\big(\mathcal{M}^{1-q}\frac{t}{r^p}\big)^{\frac{1}{2-p}}$ and choose  $\beta$ such that
$$1<\beta=\frac{1}{2}\big(1+\min (p, \frac{1}{p-1})\big) <\min\big(p, \frac{1}{p-1}\big).$$
By the H\"{o}lder inequality we obtain
\begin{multline}\label{eq6.3}
\int\limits^t_0\fint\limits_{B_{r}(y)}|D u|^{p-1} \zeta^{l-1}(x)\,dx\,d\tau\leqslant
\bigg(\int\limits^t_0\fint\limits_{B_{r}(y)}\Big(\frac{\tau}{t}\Big)^{\frac{1}{2(p-1)}} \frac{|D u|^{p}\,\,\zeta^{l+p}(x)}{\big[\pm(\mu^{\pm}-u)+\epsilon\big]^{\beta}} \,dx\,d\tau\bigg)^{\frac{p-1}{p}}\times\\\times
\bigg(\int\limits^t_0\fint\limits_{B_{r}(y)}\Big(\frac{t}{\tau}\Big)^{\frac{1}{2}}\Big[\pm(\mu^{\pm}-u)+\epsilon\Big]^
{\beta (p-1)}\zeta^{l-p^2}(x)\,dx d\tau\bigg)^{\frac{1}{p}}.
\end{multline}
By the evident fact that $\max(|\mu^+|, |\mu^-|)=\max(\mu^+, -\mu^-)$, the algebraic lemma, Lemma \ref{lem2.1} yields
\begin{equation}\label{eq6.4}
\pm(\mu^{\pm}-u)=\pm(\mu^{\pm}-u)[|u|+|\mu^{\pm}|]^{q-1}[|u|+|\mu^{\pm}|]^{1-q}\leqslant \gamma \mathcal{M}^{1-q}\big[\pm\big(|\mu^{\pm}|^{q-1}\mu^{\pm}-|u|^{q-1}u\big)\big],
\end{equation}
so, using the H\"{o}lder inequality and the fact that $\beta (p-1)< 1$,
we estimate the second integral on the right-hand side of \eqref{eq6.3} as follows
\begin{multline}\label{eq6.5}
\int\limits^t_0\fint\limits_{B_{ r}(y)}\Big(\frac{t}{\tau}\Big)^{\frac{1}{2}}\Big[\pm(\mu^{\pm}-u)+\epsilon\Big]^{\beta(p-1)}\zeta^{l-p^2}(x)\,dx d\tau\leqslant\\
\leqslant\gamma \int\limits^t_0\fint\limits_{B_{r}(y)}\Big(\frac{t}{\tau}\Big)^{\frac{1}{2}}\Big[\pm(|\mu^{\pm}|^{q-1}\mu^{\pm}-|u|^{q-1}u)\,
\mathcal{M}^{1-q}+\epsilon\Big]^{\beta(p-1)}\zeta^{l-p^2}(x)\,dx d\tau\leqslant\\\leqslant \gamma\, t\,\mathcal{M}^{(1-q)\beta(p-1)}\,\sup\limits_{0\leqslant \tau\leqslant t}\fint\limits_{B_{r}(y)\times\{\tau\}}
\Big[\pm\big(|\mu^{\pm}|^{q-1}\mu^{\pm}-|u|^{q-1}u\big)+\epsilon \mathcal{M}^{q-1}\Big]^{\beta(p-1)}\zeta^{l-p^2}(x)\,dx \leqslant\\
\leqslant \gamma\,t\,\mathcal{M}^{(1-q)\beta(p-1)}\bigg(\sup\limits_{0\leqslant \tau\leqslant t}\fint\limits_{B_{r}(y)\times\{\tau\}}
\Big[\pm\big(|\mu^{\pm}|^{q-1}\mu^{\pm}-|u|^{q-1}u\big)+\epsilon \mathcal{M}^{q-1}\Big]\zeta^{\frac{l-p^2}{\beta(p-1)}}(x)\,dx \bigg)^{\beta(p-1)}\leqslant\\\leqslant  
\gamma\,t\,\mathcal{M}^{(1-q)\beta(p-1)}\bigg(\sup\limits_{0\leqslant \tau\leqslant t}\fint\limits_{B_{r}(y)\times\{\tau\}}
\Big[\pm\big(|\mu^{\pm}|^{q-1}\mu^{\pm}-|u|^{q-1}u\big)+\epsilon \mathcal{M}^{q-1}\Big]\zeta^{l}(x)\,dx \bigg)^{\beta(p-1)}
\leqslant\\\leqslant
\gamma\,t\,\mathcal{M}^{(1-q)\beta(p-1)}\Big(\sup\limits_{0\leqslant \tau\leqslant t}\mathcal{I}(\tau)+\epsilon \mathcal{M}^{q-1}\Big)^{\beta(p-1)},
\end{multline}
provided that
$$\frac{l-p^2}{\beta(p-1)}\geqslant l,\quad \text{i. e.}\quad l\geqslant \frac{2 p^2}{2-p}>\frac{p^2}{1-\beta(p-1)}.$$
To estimate the first integral on the right-hand side of \eqref{eq6.3} we test \eqref{eq2.1} by
$$\Big(\frac{\tau}{t}\Big)^{\frac{1}{2(p-1)}} \big[\pm(\mu^{\pm}-u)+\epsilon\big]^{1-\beta}\,\zeta^{l+p}(x),$$
the use of such a test function is justified,  by making use of the
exponential time mollification  $\llbracket u \rrbracket_h$( see \cite[Chapter 7.1]{Bog3}  for non-negative solutions, the sign-changing case can be considered completely analogous).  By the Young inequality, conditions \eqref{eq1.2}, \eqref{eq6.4} and using the fact that $p-1< 1<
\beta <p$ we obtain
\begin{multline}\label{eq6.6}
\gamma\,\int\limits^t_0\fint\limits_{B_{r}(y)}\Big(\frac{\tau}{t}\Big)^{\frac{1}{2(p-1)}} \frac{|D u|^{p}\zeta^{l+p}(x)}{\big[\pm(\mu^{\pm}-u)+\epsilon\big]^{\beta}} \,dx\,d\tau+\\+\frac{q}{2(p-1)}\int\limits^t_0\fint\limits_{B_r(y)}
\frac{\tau^{\frac{1}{2(p-1)}-1}}{t^{\frac{1}{2(p-1)}}}
\int\limits_0^{\pm(\mu^{\pm}-u)}\frac{|z\mp \mu^{\pm}|^{q-1}}{\big[z+\epsilon\big]^{\beta-1}}\,dz\,\zeta^{l+p}(x) dx\,d\tau
\leqslant\\\leqslant \gamma(l) \fint\limits_{B_{r}(y)\times\{t\}}
\int\limits_0^{\pm(\mu^{\pm}-u)}\frac{|z\mp \mu^{\pm}|^{q-1}}{\big[z+\epsilon\big]^{\beta-1}}\,dz\,\zeta^{l+p}(x) dx+\\+
\frac{\gamma(l)}{(\sigma r)^p}\int\limits^t_0\fint\limits_{B_{r}(y)}\Big(\frac{\tau}{t}\Big)^{\frac{1}{2(p-1)}}\big[\pm(\mu^{\pm}-u)+\epsilon\big]^{p-\beta}\zeta^{l}(x)
\,dx\,d\tau
\leqslant\\\leqslant \frac{\gamma(l)}{\epsilon^{\beta-1}}\fint\limits_{B_{r}(y)\times\{t\}}
\int\limits_0^{\pm(\mu^{\pm}-u)}|z\mp \mu^{\pm}|^{q-1}\,dz\zeta^l(x)\, dx+\\+\frac{\gamma(l)\,t}{(\sigma r)^p}\sup\limits_{0\leqslant \tau\leqslant t}\fint\limits_{B_{r}(y)\times\{\tau\}}\big[\pm(\mu^{\pm}-u)+\epsilon\big]^{p-\beta}\zeta^l(x)\,dx
\leqslant \\\leqslant  \frac{\gamma(l)}{\epsilon^{\beta-1}}\sup\limits_{0\leqslant \tau\leqslant t}\fint\limits_{B_{r}(y)\times\{\tau\}}
\big[\pm(|\mu^{\pm}|^{q-1}\mu^{\pm}-|u|^{q-1}u)\big]\zeta^l(x)\,dx+\\+\frac{\gamma(l)\,t}{(\sigma r)^p}\mathcal{M}^{(1-q)(p-\beta)}\sup\limits_{0\leqslant \tau\leqslant t}\fint\limits_{B_{r}(y)\times\{\tau\}}
\big[\pm(|\mu^{\pm}|^{q-1}\mu^{\pm}-|u|^{q-1}u)+\epsilon \mathcal{M}^{q-1}\big]^{p-\beta}\zeta^l(x)\,dx\leqslant 
\\\leqslant \frac{\gamma(l)}{\epsilon^{\beta-1}}\sup\limits_{0\leqslant \tau\leqslant t}\fint\limits_{B_{r}(y)\times\{\tau\}}
\big[\pm(|\mu^{\pm}|^{q-1}\mu^{\pm}-|u|^{q-1}u)\big]\zeta^l(x)\,dx+\\+\frac{\gamma(l)\,t}{(\sigma r)^p}\frac{\mathcal{M}^{1-q}}{\epsilon^{\beta-p+1}}\sup\limits_{0\leqslant \tau\leqslant t}\fint\limits_{B_{r}(y)\times\{\tau\}}
\big[\pm(|\mu^{\pm}|^{q-1}\mu^{\pm}-|u|^{q-1}u)+\epsilon \mathcal{M}^{q-1}\big]\zeta^l(x)\,dx\leqslant\\\leqslant
\frac{\gamma(l)}{\epsilon^{\beta-1}\sigma^p}\sup\limits_{0\leqslant \tau\leqslant t}\fint\limits_{B_{r}(y)\times\{\tau\}}
\big[\pm(|\mu^{\pm}|^{q-1}\mu^{\pm}-|u|^{q-1}u)+\epsilon \mathcal{M}^{q-1}\big]\zeta^l(x)\,dx\leqslant\\\leqslant\frac{\gamma(l)}{ \epsilon^{\beta-1}\sigma^p}\Big(\sup\limits_{0\leqslant \tau\leqslant t}\mathcal{I}(\tau)+\epsilon \mathcal{M}^{q-1}\Big).
\end{multline}
Combining \eqref{eq6.5}, \eqref{eq6.6}, from \eqref{eq6.3} we obtain
\begin{multline*}
\frac{\gamma l}{\sigma r}\int\limits^t_0\fint\limits_{B_{r}(y)}|D u|^{p-1} \zeta^{l-1}(x)\,dx\,d\tau\leqslant\\\leqslant
\frac{\gamma(l)}{\sigma^p \epsilon^{\frac{(\beta-1)(p-1)}{p}}}\Big(\frac{t}{r^p}\Big)^{\frac{1}{p}}\mathcal{M}^{\frac{(1-q)\beta(p-1)}{p}}
\Big(\sup\limits_{0\leqslant \tau\leqslant t}\mathcal{I}(\tau)+\epsilon \mathcal{M}^{q-1}\Big)^{\frac{(1+\beta)(p-1)}{p}}=\\=
\frac{\gamma(l)}{\sigma^p}\big(\epsilon \mathcal{M}^{q-1}\big)^{\frac{1-\beta(p-1)}{p}}\Big(\sup\limits_{0\leqslant \tau\leqslant t}\mathcal{I}(\tau)+\epsilon \mathcal{M}^{q-1}\Big)^{\frac{(1+\beta)(p-1)}{p}},
\end{multline*}
using the Young inequality we get
\begin{equation*}
\frac{\gamma l}{\sigma r}\int\limits^t_0\fint\limits_{B_{ r}(y)}|D u|^{p-1} \zeta^{l-1}(x)\,dx\,d\tau\leqslant
\delta\,\sup\limits_{0\leqslant \tau\leqslant t}\mathcal{I}(\tau)+\frac{\gamma(l)}{\delta^\gamma \sigma^\gamma}\,\epsilon\,\mathcal{M}^{q-1},
\end{equation*}
from which we arrive at the required \eqref{eq6.2}, this completes the proof of the lemma.
\end{proof}

\subsubsection{Proof of Theorem \ref{th1.2} Concluded}
Choosing $t_1$, $t_2$ such that
$$\fint\limits_{B_{ r}(y)\times\{t_1\}}\pm\big(|\mu^{\pm}|^{q-1}\mu^{\pm}-|u|^{q-1}u\big)\zeta^l(x)\,dx=\sup\limits_{0\leqslant \tau\leqslant t}\mathcal{I}(\tau),$$
and
$$\fint\limits_{B_{r}(y)\times\{t_2\}}\pm\big(|\mu^{\pm}|^{q-1}\mu^{\pm}-|u|^{q-1}u\big)\zeta^l(x)\,dx=
\inf\limits_{0\leqslant \tau\leqslant t}\fint\limits_{B_{r}(y)\times\{\tau\}}\pm\big(|\mu^{\pm}|^{q-1}\mu^{\pm}-|u|^{q-1}u\big)\zeta^l(x)\,dx,$$
from \eqref{eq6.1} and Lemma \ref{lem6.1} we obtain with any $\delta \in (0, 1)$
\begin{multline*}
(1-\delta)\,\sup\limits_{0\leqslant \tau\leqslant t}\mathcal{I}(\tau)\leqslant \inf\limits_{0\leqslant \tau\leqslant t}\fint\limits_{B_{r}(y)\times\{\tau\}}\pm\big(|\mu^{\pm}|^{q-1}\mu^{\pm}-|u|^{q-1}u\big)\zeta^l(x)\,dx+\\+
\frac{\gamma(l)}{\delta^\gamma \sigma^\gamma}
\big[\max(\mu^+, -\mu^-)\big]^{\frac{(1-q)(p-1)}{2-p}}
\Big(\frac{t}{r^p}\Big)^{\frac{1}{2-p}},
\end{multline*}
from which  the required \eqref{eq1.8} follows, this completes the proof of Theorem \ref{th1.2}.

\subsection{Proof of Theorem \ref{th1.3}, Case $q>p-1$ and $q>1$}

As in the previous Section we will estimate the second integral on the right-hand side of \eqref{eq6.1}.
\begin{lemma}\label{lem6.2}
For any $\delta \in (0, 1)$ there holds
\begin{equation}\label{eq6.7}
\frac{\gamma l}{\sigma r}\int\limits^t_0\fint\limits_{B_{r}(y)}|D u|^{p-1} \zeta^{l-1}(x)\,dx\,d\tau\leqslant \delta\,\sup\limits_{0\leqslant \tau\leqslant t}\mathcal{I}(\tau)+\frac{\gamma}{\delta^\gamma \sigma^\gamma}
\Big(\frac{t}{r^p}\Big)^{\frac{q}{q-p+1}},
\end{equation}
where
$$\mathcal{I}(\tau):=\fint\limits_{B_{r}(y)\times\{\tau\}}\pm(|\mu^{\pm}|^{q-1}\mu^{\pm}-|u|^{q-1}u)\zeta^l(x)\,dx.$$
\end{lemma}
\begin{proof}
Since $q> 1$, by the algebraic lemma, Lemma \ref{lem2.1} there holds
\begin{equation}\label{eq6.8}
[\pm(\mu^{\pm}-u)]^q\leqslant \pm(\mu^{\pm}-u) (|\mu^{\pm}|+|u|)^{q-1}\leqslant \gamma [|\mu^{\pm}|^{q-1}\mu^{\pm}-|u|^{q-1}u].
\end{equation}
Choose  $\epsilon=\big(\frac{t}{r^p}\big)^{\frac{1}{q-p+1}}$ and fix $\beta$ such that
$$1<\beta=\frac{1}{2}\big(1+\min(p, \frac{q}{p-1})\big)<\min\big(p, \frac{q}{p-1}\big).$$
By the H\"{o}lder inequality we obtain
\begin{multline}\label{eq6.9}
\int\limits^t_0\fint\limits_{B_{r}(y)}|D u|^{p-1}\,\zeta^{l-1}(x)\,dx d\tau\leqslant\\\leqslant\Big(\int\limits^t_0\fint\limits_{B_{ r}(y)}\Big(\frac{\tau}{t}\Big)^{\frac{1}{2(p-1)}}\,\big[\pm(\mu^{\pm}-u)+\epsilon]^{-\beta}|D u|^{p}\,\zeta^{l+p}(x)\,dx d\tau\Big)^{\frac{p-1}{p}}\times\\\times\Big(
\int\limits^t_0\fint\limits_{B_{r}(y)}\Big(\frac{t}{\tau}\Big)^{\frac{1}{2}}
\big[\pm(\mu^{\pm}-u)+\epsilon\big]^{\beta(p-1)}\zeta^{l-p^2}(x)\,dx\,d\tau\Big)^{\frac{1}{p}}.
\end{multline}
By the H\"{o}lder inequality, \eqref{eq6.8}, using the fact that $\frac{\beta (p-1)}{q} <1$ and choosing $l$ such that
$$\frac{q(l-p^2)}{\beta (p-1)}\geqslant l,\quad \text{i. e.}\quad l\geqslant \frac{2 q p^2}{q-p+1}>\frac{q p^2}{q-\beta(p-1)},$$
we obtain
\begin{multline}\label{eq6.10}
\int\limits^t_0\fint\limits_{B_{r}(y)}\Big(\frac{t}{\tau}\Big)^{\frac{1}{2}}
\big[\pm(\mu^{\pm}-u)+\epsilon\big]^{\beta(p-1)}\zeta^{l-p^2}(x)\,dx\,d\tau\leqslant\\\leqslant
\gamma \int\limits^t_0\fint\limits_{B_{r}(y)}\Big(\frac{t}{\tau}\Big)^{\frac{1}{2}}
\Big[[\pm(\mu^{\pm}-u)]^q+\epsilon^q\Big]^{\frac{\beta (p-1)}{q}}\zeta^{l-p^2}(x)\,dx\,d\tau\leqslant\\\leqslant
\gamma t\,\Big(\sup\limits_{0\leqslant \tau\leqslant t}\fint\limits_{B_{ r}(y)\times\{\tau\}}\big[\pm(|\mu^{\pm}|^{q-1}\mu^{\pm}-|u|^{q-1}u)+\epsilon^q\big]\zeta^{\frac{q (l-p^2)}{\beta (p-1)}}(x)\,dx\Big)
^{\frac{\beta (p-1)}{q}}\leqslant\\\leqslant
\gamma t\,\Big(\sup\limits_{0\leqslant \tau\leqslant t}\fint\limits_{B_{ r}(y)\times\{\tau\}}\big[\pm(|\mu^{\pm}|^{q-1}\mu^{\pm}-|u|^{q-1}u)+\epsilon^q\big]\zeta^l(x)\,dx\Big)
^{\frac{\beta (p-1)}{q}}\leqslant\\\leqslant
\gamma t\Big(\sup\limits_{0\leqslant \tau\leqslant t}\mathcal{I}(\tau)+\epsilon^q\Big)^{\frac{\beta (p-1)}{q}}.
\end{multline}
To estimate the first term on the right-hand side of \eqref{eq6.9} we test identity \eqref{eq2.1} by $$\Big(\frac{\tau}{t}\Big)^{\frac{1}{2(p-1)}}[\pm(\mu^{\pm}-u)+\epsilon]^{1-\beta}\zeta^{l+p}(x),$$ 
the use of such a test function is justified,  by making use of the
exponential time mollification  $\llbracket u \rrbracket_h$.
By the Young inequality, \eqref{eq1.2}, \eqref{eq6.8} and using the fact that $p-q< 1< \beta <p$ we obtain
\begin{multline}\label{eq6.11}
\gamma\int\limits^t_0\fint\limits_{B_{r}(y)}\Big(\frac{\tau}{t}\Big)^{\frac{1}{2(p-1)}}[\pm(\mu^{\pm}-u)+\epsilon]^{-\beta}|D u|^{p}\,\zeta^{l+p}(x)\,dx d\tau+\\+\frac{q}{2(p-1)}\int\limits^t_0\fint\limits_{B_r(y)}
\frac{\tau^{\frac{1}{2(p-1)}-1}}{t^{\frac{1}{2(p-1)}}}
\int\limits_0^{\pm(\mu^{\pm}-u)}\frac{|z\mp \mu^{\pm}|^{q-1}}{\big[z+\epsilon\big]^{\beta-1}}\,dz\,\zeta^{l+p}(x) dx\,d\tau
\leqslant \\\leqslant \gamma \fint\limits_{B_{ r}(y)\times\{t\}}\int\limits_{0}^{\pm(\mu^{\pm}-u)}\frac{|z\mp\mu^{\pm}|^{q-1}}{\big[z+\epsilon\big]^{\beta-1}}\,dz\zeta^{l+p}(x)\,dx+\\+
\frac{\gamma(l)}{(\sigma r)^p} \int\limits^t_0\fint\limits_{B_{r}(y)}\big[\pm(\mu^{\pm}-u)+\epsilon]^{p-\beta}\zeta^l(x)\,dx d\tau\leqslant \\\leqslant
\frac{\gamma(l)}{\epsilon^{\beta-1}}\,\sup\limits_{0\leqslant \tau\leqslant t}
\fint\limits_{B_{r}(y)\times\{\tau\}}\pm(|\mu^{\pm}|^{q-1}\mu^{\pm}-|u|^{q-1}u)\zeta^l(x)\,dx 
+\\+\frac{\gamma(l)\,t}{(\sigma r)^p}\sup\limits_{0\leqslant \tau\leqslant t}\,\fint\limits_{B_{ r}(y)\times\{\tau\}}\big[\pm(|\mu^{\pm}|^{q-1}\mu^{\pm}-|u|^{q-1}u)+\epsilon^q\big]^{\frac{p-\beta}{q}}\zeta^l(x)
\,dx
\leqslant\\\leqslant
\frac{\gamma(l)}{\epsilon^{\beta-1}}\,\sup\limits_{0\leqslant \tau\leqslant t}
\fint\limits_{B_{r}(y)\times\{\tau\}}\pm(|\mu^{\pm}|^{q-1}\mu^{\pm}-|u|^{q-1}u)\zeta^l(x)\,dx 
+\\+\frac{\gamma(l)\,t}{(\sigma r)^p\,\epsilon^{\beta+q-p}}\sup\limits_{0\leqslant \tau\leqslant t}\,\fint\limits_{B_{ r}(y)\times\{\tau\}}\big[\pm(|\mu^{\pm}|^{q-1}\mu^{\pm}-|u|^{q-1}u)+\epsilon^q\big]\zeta^l(x)
\,dx\leqslant\\\leqslant
\frac{\gamma(l)}{\sigma^p \epsilon^{\beta-1}}\sup\limits_{0\leqslant \tau\leqslant t}
\fint\limits_{B_{r}(y)\times\{\tau\}}\big[\pm(|\mu^{\pm}|^{q-1}\mu^{\pm}-|u|^{q-1}u)+\epsilon^q\big]\zeta^l(x)\,dx\leqslant\\\leqslant \frac{\gamma(l)}{\sigma^p \epsilon^{\beta-1}}\big(\sup\limits_{0\leqslant \tau\leqslant t}\mathcal{I}(\tau) +\epsilon^q\big),
\end{multline}
Combining \eqref{eq6.9}-\eqref{eq6.11} and using the Young inequality we arrive at
\begin{multline*}
\frac{\gamma\,l}{\sigma r}\int\limits^t_0\fint\limits_{B_{ r}(y)}|D u|^{p-1}\,\zeta^{l-1}(x)\,dx d\tau\leqslant
\frac{\gamma(l)}{\sigma^p \epsilon^{\frac{(\beta-1)(p-1)}{p}}}\Big(\frac{t}{r^p}\Big)^{\frac{1}{p}}\Big(
\sup\limits_{0\leqslant \tau\leqslant t}\mathcal{I}(\tau)+
\epsilon^q\Big)^{\frac{p-1}{p}(1+\frac{\beta}{q})}=\\=
\frac{\gamma(l)}{\sigma^p}\,\epsilon^{\frac{q-\beta(p-1)}{p}}\Big(
\sup\limits_{0\leqslant \tau\leqslant t}\mathcal{I}(\tau)+
\epsilon^q\Big)^{\frac{p-1}{p}(1+\frac{\beta}{q})}\leqslant \delta \sup\limits_{0\leqslant \tau\leqslant t}\mathcal{I}(\tau)+
\frac{\gamma(l)}{\delta^\gamma \sigma^\gamma}\,\epsilon^q,
\end{multline*}
which completes the proof of the lemma.
\end{proof}

\subsubsection{Proof of Theorem \ref{th1.3} Concluded}
Choosing $t_1$, $t_2$ as in the previous section, inequality \eqref{eq6.1} and Lemma \ref{lem6.2} imply
\begin{equation*}
(1-\delta)\sup\limits_{0\leqslant \tau\leqslant t}\mathcal{I}(\tau)\leqslant 
\inf\limits_{0\leqslant \tau\leqslant t}\fint\limits_{B_{r}(y)}\pm (|\mu^{\pm}|^{q-1}\mu^{\pm}-|u|^{q-1}u)\zeta^l(x)\,dx+
\frac{\gamma(l)}{\sigma^\gamma \delta^\gamma}\Big(\frac{t}{r^p}\Big)^{\frac{q}{q-p+1}},
\end{equation*}
which yields the required \eqref{eq1.9}, this completes the proof of Theorem \ref{th1.3} in the case $q>p-1$ and $q>1$.

\section{Appendix B. Proof of Theorem \ref{th1.4}}
 Construct the cylinder
$$Q_{r, \theta}:=B_r(x_0)\times (t_0-\theta, t_0)\subset \Omega_T, \quad \theta:=b\,\omega^{q-p+1}\,r^p,$$ such that
$$\mu^+\geqslant \sup\limits_{Q_{r, \theta}}u,\quad \mu^-\leqslant \inf\limits_{Q_{r, \theta}} u,
\quad \omega\geqslant \mu^+-\mu^-.$$
The following is the key lemma in the proof of Theorem \ref{th1.4}. 
\begin{propo}\label{pr7.1}
 Let $u$ be a locally bounded, local weak sub(super)-solution to \eqref{eq1.1}, \eqref{eq1.2}
 in $\Omega_T$ and let $q>p-1>0$. Assume also that
\begin{equation}\label{eq7.1}
|B_r(x_0)\cap\big\{\pm (\mu^{\pm}-u(\cdot, t))\geqslant \xi\,\omega\big\}|\geqslant \alpha|B_r(x_0)|,
\end{equation}
 for all $t\in (t_0-\theta, t_0)$  and with some $b$, $\xi$, $\alpha \in (0, 1)$. Then for any $\nu\in (0, 1)$ 
 there exist numbers   $\epsilon_*$, $\xi_*\in (0, 1)$ depending only on the data, $\nu$, $\xi$, $\alpha$ and $b$ such that
 for all  $t\in (t_0-\frac{1}{2}\theta, t_0)$
\begin{equation}\label{eq7.2}
\big|B_{\frac{3}{4}r}(x_0)\cap\big\{\mp u(x, t)\leqslant \xi_*\,\omega\big\}\big|\leqslant \nu\,|B_{\frac{3}{4}r}(x_0)|,
\end{equation}
provided that 
\begin{equation}\label{eq7.3}
\pm \mu^{\pm}\leqslant \epsilon_*\,\omega.
\end{equation}
\end{propo}
We will  proceed  with  an  additional  regularity  assumption  that
\begin{equation}\label{eq7.4}
\frac{\partial}{\partial t}(|u|^{q-1} u) \in C(t_0-\theta, t_0; B_r(x_0)).
\end{equation}
The analysis of removing this assumption has been carried out in \cite[Section 6]{Bog3} for non-negative solutions. However, the
 same proof actually works for sign-changing solutions.

Observe that conditions \eqref{eq7.1}, \eqref{eq7.3} yield
\begin{equation}\label{eq7.5}
|B_r(x_0)\cap\big\{\mp u\geqslant \xi_0\,\omega\big\}|\geqslant \alpha|B_r(x_0)|,\quad \text{for all}\quad t\in(t_0-\theta, t_0),
\end{equation}
provided that $\epsilon_*\leqslant \xi_0:=\frac{1}{2}\,\xi.$

Define the piecewise smooth cutoff function $\zeta=\zeta_1(x) \zeta_2(t)$, where $\zeta_1(x)\in C^1_0(B_r(x_0))$, $\zeta_1(x)=1$ in $B_{\frac{3}{4} r}(x_0)$, $0\leqslant \zeta_1(x)\leqslant 1$, $|D \zeta_1(x)|\leqslant \frac{\gamma}{r}$ and $\zeta_2(t)\in C^1(\mathbb{R})$,
$\zeta_2(t)=1$ for $t\geqslant t_0-\frac{1}{2}\theta$, $\zeta_2(t)=0$ for $t\leqslant t_0-\frac{3}{4}\theta$, $0\leqslant \zeta_2(t)\leqslant 1$, $|\zeta'_t|\leqslant \frac{\gamma}{\theta}$.

The proof of Proposition \ref{pr7.1} consists of several lemmas. 

\begin{lemma}\label{lem7.1}
For   all $t\in (t_0-\theta, t_0)$, every $\epsilon \in (0, 1)$, every $j\geqslant 0$ there holds
\begin{multline}\label{eq7.6}
\frac{\partial}{\partial t}\int\limits_{B_r(x_0)}\int\limits^{(u\pm\epsilon^j\xi_0\omega)_{\pm}}_0\frac{|\epsilon^j\xi_0\omega-z|^{q-1} dz}{[\epsilon^j\xi_0\omega-z+\epsilon^{j+1}\xi_0\omega]^{p-1}}\,\zeta^p\,dx+\\+
\frac{1}{\gamma(\alpha) r^p}\,\int\limits_{B_r(x_0)}\Big[\log\frac{\epsilon^j(1+\epsilon)\xi_0\omega}{\epsilon^j\xi_0\omega-(u\pm\epsilon^j\xi_0\omega)_{\pm}+
\epsilon^{j+1}\xi_0\omega}\Big]^p\,
\zeta^p\,dx\leqslant\gamma(b)\,r^{N-p}.
\end{multline}
\end{lemma}
\begin{proof}
By Lemma \ref{lem2.7} the function $\epsilon^j\xi_0\omega-(u\pm \epsilon^j\xi_0\omega)_{\pm}$ is a local, weak super-solution to \eqref{eq1.1}, \eqref{eq1.2}, so, we test
\eqref{eq2.1} by
$$\frac{\zeta^p}{\big[\epsilon^j\xi_0\omega-(u\pm \epsilon^j\xi_0\omega)_{\pm}+\epsilon^{j+1}\xi_0\omega\big]^{p-1}}.$$
By standard calculations, using \eqref{eq1.2} and the Young inequality we obtain
\begin{multline}\label{eq7.7}
\frac{\partial}{\partial t}\int\limits_{B_r(x_0)}\int\limits^{(u\pm\epsilon^j\xi_0\omega)_{\pm}}_0\frac{|\epsilon^j\xi_0\omega-z|^{q-1} dz}{[\epsilon^j\xi_0\omega-z+\epsilon^{j+1}\xi_0\omega]^{p-1}}\,\zeta^p\,dx+\\+
\frac{1}{\gamma}\,\int\limits_{B_r(x_0)}\Big|D \log\frac{\epsilon^j(1+\epsilon)\xi_0\omega}{\epsilon^j\xi_0\omega-(u\pm\epsilon^j\xi_0\omega)_{\pm}+
\epsilon^{j+1}\xi_0\omega}\Big|^p\,
\zeta^p\,dx\leqslant\\\leqslant \frac{\gamma}{\theta}\int\limits_{B_r(x_0)}\int\limits^{(u\pm\epsilon^j\xi_0\omega)_{\pm}}_0\frac{|\epsilon^j\xi_0\omega-z|^{q-1} dz}{[\epsilon^j\xi_0\omega-z+\epsilon^{j+1}\xi_0\omega]^{p-1}}\,\zeta^{p-1}\,dx+\gamma\,r^{N-p}.
\end{multline}
Let us estimate the terms on the left and right-hand side of \eqref{eq7.7}. Choosing $\epsilon_*$ from the condition
$$\epsilon_* \leqslant \frac{1}{2}\, \epsilon^{j+1}\,\xi_0,$$
by \eqref{eq7.3} we obtain
\begin{multline*}
\int\limits^{(\epsilon^j\xi_0\omega\pm u)_{\pm}}_0\frac{|\epsilon^j\xi_0\omega-z|^{q-1} dz}{[\epsilon^j\xi_0\omega-z+\epsilon^{j+1}\xi_0\omega]^{p-1}} \leqslant
\int\limits^{(\epsilon^j\xi_0+\epsilon_*)\omega}_0\frac{|\epsilon^j\xi_0\omega-z|^{q-1} dz}{[\epsilon^j\xi_0\omega-z+\epsilon^{j+1}\xi_0\omega]^{p-1}}\leqslant\\\leqslant
\int\limits^{\epsilon^j\xi_0\omega}_0
(\epsilon^j\xi_0\omega-z)^{q-p} dz +\int\limits^{(\epsilon^j\xi_0+\epsilon_*)\omega}_{\epsilon^j\xi_0\omega}
\frac{(z-\epsilon^j\xi_0\omega)^{q-1} dz}{[\epsilon^j\xi_0\omega-z+\epsilon^{j+1}\xi_0\omega]^{p-1}}\leqslant\\\leqslant 
\gamma\,(\epsilon \xi_0 \omega)^{j(q-p+1)}+\int\limits_0^{\epsilon_* \omega}\frac{z^{q-1} dz}{[\epsilon^{j+1}\xi_0 \omega-z]^{p-1}}\leqslant
\gamma\,(\epsilon \xi_0 \omega)^{j(q-p+1)}+\gamma\,\frac{\epsilon^{q}_*\,\omega^{q-p+1}}{(\epsilon^{j+1}\xi_0-\epsilon_*)^{p-1}}
\leqslant\\\leqslant
\gamma\,(\epsilon \xi_0 \omega)^{j(q-p+1)}+\gamma\,\frac{\epsilon^{q}_*\,\omega^{q-p+1}}{(\epsilon^{j+1}\xi_0)^{p-1}}
\leqslant \gamma\,(\epsilon \xi_0 \omega)^{j(q-p+1)}\leqslant\gamma\,\omega^{q-p+1},
\end{multline*}
hence, the right-hand side of \eqref{eq7.7} is estimated by
$$\frac{\gamma}{\theta}\int\limits_{B_r(x_0)}\int\limits^{(u\pm\epsilon^j\xi_0\omega)_{\pm}}_0\frac{|\epsilon^j\xi_0\omega-z|^{q-1} dz}{[\epsilon^j\xi_0\omega-z+\epsilon^{j+1}\xi_0\omega]^{p-1}}\,\zeta^{p-1}\,dx+\gamma\,r^{N-p}\leqslant \frac{\gamma}{b}\,r^{N-p}.$$
By the Poincare type inequality, see e.g. \cite[Proposition 2.1]{DiB1}, using \eqref{eq7.1} we estimate the second term on the left-hand side 
of \eqref{eq7.7} as follows
\begin{multline*}
\int\limits_{B_r(x_0)}\Big|D \log\frac{\epsilon^j(1+\epsilon)\xi_0\omega}{\epsilon^j\xi_0\omega-(u\pm\epsilon^j\xi_0\omega)_{\pm}+
\epsilon^{j+1}\xi_0\omega}\Big|^p\,
\zeta^p\,dx\geqslant\\\geqslant \frac{1}{\gamma 
\,\alpha^p\,r^p}\int\limits_{B_r(x_0)}\Big[ \log\frac{\epsilon^j(1+\epsilon)\xi_0\omega}{\epsilon^j\xi_0\omega-(u\pm\epsilon^j\xi_0\omega)_{\pm}+
\epsilon^{j+1}\xi_0\omega}\Big]^p\,
\zeta^p\,dx.
\end{multline*}
Combining the last two inequalities, from \eqref{eq7.7} we arrive at the required \eqref{eq7.6}, this completes the proof of the lemma.
\end{proof}

\subsection{Main Proposition}
Set
\begin{equation*}
\begin{cases}
A_j(t):=B_r(x_0)\cap\big\{\mp u(\cdot, t)\leqslant \epsilon^j\,\xi_0 \omega\big\},\\
Y_j(t):=\frac{1}{|B_r(x_0)|}\int\limits_{A_j(t)}\,\zeta^p\,dx,\quad y_j:=\sup\limits_{t\in(t_0-\theta, t_0)}\,Y_j(t).
\end{cases}
\end{equation*}

Proposition \ref{pr7.1} is a simple consequence of the following result
\begin{propo}\label{pr7.2}
For any $\tilde{\nu}\in (0, 1)$ there exist $\epsilon \in (0, 1)$ and  $j_*\geqslant 1$ such that
\begin{equation}\label{eq7.8}
y_{j_*}\leqslant \tilde{\nu}.
\end{equation}
\end{propo}

Fix $j\geqslant 1$ and assume that $y_j\geqslant \tilde{\nu}$ and for any $\delta\in (0, 1)$ we find $t_\delta \in (t_0-\theta, t_0)$ such that
$$y_{j+1}\leqslant Y_{j+1}(t_\delta)+\delta.$$
The proof of Proposition \ref{pr7.2} distinguishes two different cases:
\begin{equation}\label{eq7.9}
\frac{\partial}{\partial t}\int\limits_{B_r(x_0)\times\{t_\delta\}}\int\limits^{(u\pm\epsilon^j\xi_0\omega)_{\pm}}_0\frac{|\epsilon^j\xi_0\omega-z|^{q-1} dz}{[\epsilon^j\xi_0\omega-z+\epsilon^{j+1}\xi_0\omega]^{p-1}}\,\zeta^p\,dx \geqslant 0,
\end{equation}
or
\begin{equation}\label{eq7.10}
\frac{\partial}{\partial t}\int\limits_{B_r(x_0)\times\{t_\delta\}}\int\limits^{(u\pm\epsilon^j\xi_0\omega)_{\pm}}_0\frac{|\epsilon^j\xi_0\omega-z|^{q-1} dz}{[\epsilon^j\xi_0\omega-z+\epsilon^{j+1}\xi_0\omega]^{p-1}}\,\zeta^p\,dx < 0.
\end{equation}

\subsubsection{Proof of Proposition \ref{pr7.2} Under Condition \eqref{eq7.9}}
Assume first that inequality \eqref{eq7.9} holds, then by our choices Lemma \ref{lem7.1} yields
\begin{multline*}
\frac{1}{\gamma(\alpha)}(y_{j+1}-\delta)\Big[\log\frac{1+\epsilon}{2\epsilon}\Big]^p\leqslant
\frac{1}{\gamma(\alpha)}\,Y_{j+1}(t_\delta)\,\Big[\log\frac{1+\epsilon}{2\epsilon}\Big]^p\leqslant\\\leqslant
\frac{1}{\gamma(\alpha)}\,\fint\limits_{B_r(x_0)\times\{t_\delta\}}
\Big[\log\frac{\epsilon^j(1+\epsilon)\xi_0\omega}{\epsilon^j\xi_0\omega-(u\pm\epsilon^j\xi_0\omega)_{\pm}+
\epsilon^{j+1}\xi_0\omega}\Big]^p\,
\zeta^p\,dx\leqslant \gamma(b),
\end{multline*}
and hence
\begin{equation}\label{eq7.11}
y_{j+1}\leqslant \delta+\frac{\gamma(\alpha, b)}{[\log\frac{1+\epsilon}{2\epsilon}]^p}\leqslant \tilde{\nu},
\end{equation}
provided that
$$\delta\leqslant \frac{1}{2}\,\tilde{\nu}\quad \text{and}\quad \epsilon \leqslant \frac{1}{2}\,e^{-\big(\frac{2\gamma(\alpha, b)}{\tilde{\nu}}\big)^{\frac{1}{p}}}.$$

\subsubsection{Proof of Proposition \ref{pr7.2} Under Condition \eqref{eq7.10}}
Introduce a time level
$$t_*:=\sup\Big\{t\in (t_0-\theta, t_\delta): \frac{\partial}{\partial t}\int\limits_{B_r(x_0)\times\{t\}}\int\limits^{(u\pm\epsilon^j\xi_0\omega)_{\pm}}_0\frac{|\epsilon^j\xi_0\omega-z|^{q-1} dz}{[\epsilon^j\xi_0\omega-z+\epsilon^{j+1}\xi_0\omega]^{p-1}}\,
\zeta^p\,dx \geqslant 0 \Big\},$$
by the fact that $\zeta=0$ for $t\leqslant t_0-\frac{3}{4}\theta$ such a set is non-empty and $t_*$ is well-defined.  By our choices
\begin{equation}\label{eq7.12}
I(t_\delta):=\int\limits_{B_r(x_0)\times\{t_\delta\}}\int\limits^{(u\pm\epsilon^j\xi_0\omega)_{\pm}}_0\frac{|\epsilon^j\xi_0\omega-z|^{q-1} dz}{[\epsilon^j\xi_0\omega-z+\epsilon^{j+1}\xi_0\omega]^{p-1}}\,
\,\zeta^p\,dx\leqslant I(t_*).
\end{equation}
Using inequality \eqref{eq2.3} from Lemma \ref{lem2.2} we obtain
\begin{multline}\label{eq7.13}
I(t_\delta)\geqslant \int\limits_{A_{j+1}(t_\delta)}\int\limits^{\epsilon^j\xi_0\omega(1-\epsilon)}_0\frac{|\epsilon^j\xi_0\omega-z|^{q-1} dz}{[\epsilon^j\xi_0\omega-z+\epsilon^{j+1}\xi_0\omega]^{p-1}}\,
\,\zeta^p\,dx=\\=(\epsilon^j\xi_0\omega)^{q-p+1}\int\limits_{A_{j+1}(t_\delta)}
\,\zeta^p\,dx\int\limits^{1-\epsilon}_0\frac{(1-z)^{q-1}\,dz}{[1-z+\epsilon]^{p-1}}=
(\epsilon^j\xi_0\omega)^{q-p+1}\int\limits_{A_{j+1}(t_\delta)}
\,\zeta^p\,dx\int\limits^{1}_\epsilon \frac{z^{q-1}\,dz}{[z+\epsilon]^{p-1}}=\\=
(\epsilon^j\xi_0\omega)^{q-p+1}\int\limits_{A_{j+1}(t_\delta)}
\,\zeta^p\,dx\bigg[\int\limits^{1}_0 \frac{z^{q-1}\,dz}{[z+\epsilon]^{p-1}}-\int\limits^{\epsilon}_0 \frac{z^{q-1}\,dz}{[z+\epsilon]^{p-1}}\bigg]\geqslant\\\geqslant(\epsilon^j\xi_0\omega)^{q-p+1}(1-\epsilon^{q-p+1})|B_r(x_0)|Y_{j+1}(t_\delta)\,
\int\limits^{1}_0 \frac{z^{q-1}\,dz}{[z+\epsilon]^{p-1}}\geqslant\\\geqslant(\epsilon^j\xi_0\omega)^{q-p+1}(1-\epsilon^{q-p+1})|B_r(x_0)|\Big[y_{j+1}-\delta\Big]\,
\int\limits^{1}_0 \frac{z^{q-1}\,dz}{[z+\epsilon]^{p-1}}.
\end{multline}
Let us estimate the integral on the right-hand side of \eqref{eq7.12}. By \eqref{eq7.3} and Fubini's theorem
\begin{multline}\label{eq7.14}
I(t_*)\leqslant \int\limits_{A_j(t_*)}\,\zeta^p\int\limits^{(\epsilon^j\xi_0+\epsilon_*)\omega}_0
\frac{|\epsilon^j\xi_0\omega-z|^{q-1}}{[\epsilon^j\xi_0\omega-z+\epsilon^{j+1}\xi_0\omega]^{p-1}}\,dz\,dx=\\
= \int\limits^{\epsilon^j\xi_0\omega}_0
\frac{(\epsilon^j\xi_0\omega-z)^{q-1}}{[\epsilon^j\xi_0\omega-z+\epsilon^{j+1}\xi_0\omega]^{p-1}}
\bigg(\int\limits_{B_r(x_0)\times\{t_*\}}\mathbf{I}_{\{\pm(u\pm\epsilon^j\xi_0\omega)\geqslant z\}}\,\zeta^p\,dx\bigg)\,dz+\\+
\int\limits_{A_j(t_*)}\zeta^p\,dx\int\limits^{(\epsilon^j\xi_0+\epsilon_*)\omega}_{\epsilon^j\xi_0\omega}
\frac{(z-\epsilon^j\xi_0\omega)^{q-1}\,dz}{[\epsilon^j\xi_0\omega-z+\epsilon^{j+1}\xi_0\omega]^{p-1}}\leqslant\\\leqslant
(\epsilon^j\xi_0 \omega)^{q-p+1}\int\limits^1_0\frac{(1-z)^{q-1}}
{[1-z+\epsilon]^{p-1}}\bigg(\int\limits_{B_r(x_0)\times\{t_*\}}\mathbf{I}_{\{\pm(u\pm\epsilon^j\xi_0\omega)\geqslant z\,\epsilon^j\xi_0\omega \}}\,\zeta^p
\,dx\bigg)\,dz+\\+
(\epsilon^j\xi_0 \omega)^{q-p+1}|B_r(x_0)|\,y_j\,\int\limits^{1+\frac{\epsilon_*}{\epsilon^j\xi_0}}_1
\frac{(z-1)^{q-1}\,dz}{[1-z+\epsilon]^{p-1}}.
\end{multline}
The integrals on the right-hand side of \eqref{eq7.14} we estimate as follows. By \eqref{eq2.4} of Lemma \ref{lem2.2}
\begin{multline*}
\int\limits^{1+\frac{\epsilon_*}{\epsilon^j\xi_0}}_1
\frac{(z-1)^{q-1}\,dz}{[1-z+\epsilon]^{p-1}}=\int\limits^{\frac{\epsilon_*}{\epsilon^j\xi_0}}_0\frac{z^{q-1}\,dz}{[\epsilon-z]^{p-1}}\leqslant
\Big(\frac{\epsilon_*}{\epsilon^j\xi_0}\Big)^q\Big(\frac{1+\epsilon}{\epsilon-\epsilon_*/(\epsilon^j\xi_0)}\Big)^{p-1}
\int\limits^{1}_0\frac{z^{q-1}\,dz}{[z+\epsilon]^{p-1}}\leqslant\\\leqslant \Big(\frac{4}{\epsilon}\Big)^{p-1}\Big(\frac{\epsilon_*}{\epsilon^j\xi_0}\Big)^q
\int\limits^{1}_0\frac{z^{q-1}\,dz}{[z+\epsilon]^{p-1}}\leqslant \Big(\frac{\epsilon_*}{\epsilon^{j+1}\xi_0}\Big)^q
\int\limits^{1}_0\frac{z^{q-1}\,dz}{[z+\epsilon]^{p-1}},
\end{multline*}
provided that 
$$4^{p-1}\,\epsilon^{q-p+1}\leqslant 1 \quad \text{and}\quad \epsilon_*\leqslant \frac{1}{2}\xi_0\epsilon^{j_*+1}.$$
To estimate the first term on the right-hand side of \eqref{eq7.14} we note that
by Lemma \ref{lem7.1} and our choice of $t_*$
\begin{equation*}
\int\limits_{B_r(x_0)\times\{t_*\}}\mathbf{I}_{\{\pm(u\pm\epsilon^j\xi_0\omega)\geqslant z\,\epsilon^j\xi_0\omega \}}\,\zeta^p\,dx\leqslant\frac{\gamma(\alpha, b)}{[\log\frac{1+\epsilon}{1-z+\epsilon}]^p}|B_r(x_0)|,
\end{equation*}
for all $0<z<1$. Particularly, if $z_*=(1+\epsilon)(1-e^{-(\frac{2\gamma(\alpha, b)}{\tilde{\nu}})^{\frac{1}{p}}})\leqslant z <1$, then
\begin{equation*}
\int\limits_{B_r(x_0)\times\{t_*\}}\mathbf{I}_{\{\pm(u\pm\epsilon^j\xi_0\omega)\geqslant z\,\epsilon^j\xi_0\omega \}}\,\zeta^p\,dx\leqslant \frac{1}{2}\,\tilde{\nu}\,|B_r(x_0)|.
\end{equation*}
Hence,  by \eqref{eq2.2} and using the fact that $y_j\geqslant \tilde{\nu}$, from \eqref{eq7.14} we obtain
\begin{multline}\label{eq7.15}
I(t_*)\leqslant(\epsilon^j\xi_0\omega)^{q-p+1}|B_r(x_0)|\,\bigg[y_j\,\int\limits^{z_*}_0\frac{(1-z)^{q-1}\,dz}{[1-z+\epsilon]^{p-1}}+
\frac{1}{2}\tilde{\nu}
\int\limits^{1}_{z_*}\frac{(1-z)^{q-1}\,dz}{[1-z+\epsilon]^{p-1}}+\\+
y_j\,\Big(\frac{\epsilon_*}{\epsilon^{j+1}\xi_0}\Big)^q
\int\limits^{1}_0\frac{z^{q-1}\,dz}{[z+\epsilon]^{p-1}}\bigg]\leqslant\\\leqslant (\epsilon^j\xi_0\omega)^{q-p+1}|B_r(x_0)|\,y_j\bigg[\int\limits^{z_*}_0\frac{(1-z)^{q-1}\,dz}{[1-z+\epsilon]^{p-1}}+\frac{1}{2}
\int\limits^{1}_{z_*}\frac{(1-z)^{q-1}\,dz}{[1-z+\epsilon]^{p-1}}+\\+
\Big(\frac{\epsilon_*}{\epsilon^{j+1}\xi_0}\Big)^q
\int\limits^{1}_0\frac{z^{q-1}\,dz}{[z+\epsilon]^{p-1}}\bigg]=\\=(\epsilon^j\xi_0\omega)^{q-p+1}|B_r(x_0)|\,y_j\bigg[
\int\limits^{1}_{1-z_*}\frac{z^{q-1}\,dz}{[z+\epsilon]^{p-1}}+\frac{1}{2}\int\limits^{1-z_*}_{0}\frac{z^{q-1}\,dz}{[z+\epsilon]^{p-1}}+
\Big(\frac{\epsilon_*}{\epsilon^{j+1}\xi_0}\Big)^q
\int\limits^{1}_0\frac{z^{q-1}\,dz}{[z+\epsilon]^{p-1}}\bigg]=\\=(\epsilon^j\xi_0\omega)^{q-p+1}|B_r(x_0)|\,y_j\bigg[
\int\limits^{1}_{0}\frac{z^{q-1}\,dz}{[z+\epsilon]^{p-1}}-\frac{1}{2}\int\limits^{1-z_*}_{0}\frac{z^{q-1}\,dz}{[z+\epsilon]^{p-1}}+
\Big(\frac{\epsilon_*}{\epsilon^{j+1}\xi_0}\Big)^q
\int\limits^{1}_0\frac{z^{q-1}\,dz}{[z+\epsilon]^{p-1}}\bigg]\leqslant \\\leqslant (\epsilon^j\xi_0\omega)^{q-p+1}|B_r(x_0)|\,y_j
\int\limits^{1}_0\frac{z^{q-1}\,dz}{[z+\epsilon]^{p-1}}\bigg[1-\frac{1}{2}(1-z_*)^q+
\Big(\frac{\epsilon_*}{\epsilon^{j+1}\xi_0}\Big)^q\bigg]\leqslant \\\leqslant \Big[1-\frac{1}{4}(1-z_*)^q\Big]\,(\epsilon^j\xi_0\omega)^{q-p+1}|B_r(x_0)|\,y_j
\int\limits^{1}_0\frac{z^{q-1}\,dz}{[z+\epsilon]^{p-1}},
\end{multline}
provided that $\big(\frac{\epsilon_*}{\epsilon^{j+1}\xi_0}\big)^q\leqslant \frac{1}{4}(1-z_*)^q$,
that is $\epsilon_*\leqslant \frac{1}{4^{1+\frac{1}{q}}}\epsilon^{j_*+1}\xi_0 e^{-(\frac{2\gamma(\alpha, b)}{\tilde{\nu}})^{\frac{1}{p}}}\leqslant \frac{1}{4^{\frac{1}{q}}}\epsilon^{j_*+1}\xi_0 (1-z_*)$.

Collecting \eqref{eq7.11}-\eqref{eq7.15} we obtain that either
\begin{equation}\label{eq7.16}
y_{j+1}\leqslant \tilde{\nu},
\end{equation}
or
\begin{equation}\label{eq7.17}
y_{j+1}\leqslant \frac{1-\frac{1}{4}(1-z_*)^q}{1-\epsilon^{q-p+1}}\,y_j+\delta\leqslant (1-\epsilon^{q-p+1}+\frac{\delta}{\tilde{\nu}})\,y_j
\leqslant (1-\frac{1}{2}\epsilon^{q-p+1})\,y_j,
\end{equation}
provided that  $\delta\leqslant \frac{\tilde{\nu}}{2}\epsilon^{q-p+1}$ and $\frac{1-\frac{1}{4}(1-z_*)^q}{1-\epsilon^{q-p+1}} \leqslant (1-\epsilon^{q-p+1})$, i. e. $\epsilon+\epsilon^{\frac{q-p+1}{q}}\leqslant \frac{1}{8^{\frac{1}{q}}}\,e^{-(\frac{2\gamma(\alpha, b)}{\nu})^{\frac{1}{p}}}.$

\subsubsection{Proof of Proposition \ref{pr7.2} Concluded}

Iterating  inequality \eqref{eq7.17} if needed, we choose $j_*$ so large that $(1-\frac{1}{2}\,\epsilon^{q-p+1})^{j_*}\leqslant \tilde{\nu},$ from which the required inequality \eqref{eq7.8}
follows, this completes the proof of the proposition.

\subsubsection{Proof of Proposition \ref{pr7.1} and Theorem \ref{th1.4} Concluded}

Inequality \eqref{eq7.8} yields
\begin{equation*}
 |B_{\frac{3}{4}r}(x_0)\cap\big\{\mp u(\cdot, t)\leqslant \epsilon^{j_*}\xi_0\omega\big\}|\leqslant  \Big(\frac{4}{3}\Big)^N\,\tilde{\nu}\,|B_{\frac{3}{4}r}(x_0)|:=\nu |B_{\frac{3}{4} r}(x_0)|,
\end{equation*}
for all $t\in (t_0-\frac{1}{2}\,b\,\omega^{q-p+1}\,r^p, t_0)$,  this proves Proposition \ref{pr7.1}.

Choosing $\epsilon_*$ smaller if necessary, by De Giorgi type lemma, Lemma \ref{lem2.5} we obtain
$$\mp u(x, t)\geqslant \frac{1}{2}\epsilon^{j_*}\,\xi_0\,\omega,$$
for all $(x, t)\in B_{\frac{1}{2} r}(x_0)\times (t_0-\frac{1}{4}\,b\,\omega^{q-p+1}\,r^p, t_0).$ This completes the proof of Theorem \ref{th1.4}.

\vskip3.5mm
{\bf Acknowledgements.} This work is partially supported  the grant "Mathematical modelling of complex systems and processes related to security" of National Academy of Sciences of Ukraine under the budget programme "Support for the development of priority areas of scientific research" for 2025-2026, p/n 0125U000299, and by a grant from the Simons Foundation (SFI-PD-Ukraine-00017674, I. Skrypnik).

\bigskip

CONTACT INFORMATION

\medskip
\textbf{Igor I.~Skrypnik}\\Institute of Applied Mathematics and Mechanics,
National Academy of Sciences of Ukraine, \\ \indent Batiouk Str. 19, 84116 Sloviansk, Ukraine,\quad ihor.skrypnik@gmail.com

\end{document}